\documentclass[hidelinks,onefignum,onetabnum]{siamart220329}
\usepackage[utf8]{inputenc}

\usepackage{lipsum}
\usepackage{amsmath,bm}
\usepackage{amsfonts}
\usepackage{graphicx}
\usepackage{psfrag}
\usepackage{tikz}
\usetikzlibrary{calc,positioning}
\definecolor{corange}{rgb}{0.93, 0.57, 0.13}
\usepackage{simplewick}
\usepackage{mathrsfs}
\usepackage{url,hyperref}
\usepackage{setspace}
 \usepackage{listings}%
\usepackage{diagbox}
 
\usepackage{algorithm}
\usepackage{algorithmic}

\usepackage{amssymb}
\usepackage{booktabs}
\usepackage{comment}
\usepackage{subfigure}
 \usepackage{float}
 \usepackage{cite}
 \usepackage{epstopdf}
 \ifpdf
  \DeclareGraphicsExtensions{.eps,.pdf,.png,.jpg}
\else
  \DeclareGraphicsExtensions{.eps}
\fi

\newcommand{\bs}[1]{\boldsymbol{#1}}
\allowdisplaybreaks[4]

\def \bx{\bs x}
\def \by{\bs y}
\def \bz{\bs z}
\def \d{{\rm d}}
\def \R{\mathbb{R}}
\def \X{\widetilde{X}}
\def \Y{\widetilde{Y}}
\def \Z{\widetilde{Z}}
\def \E{\mathbb{E}}

\usepackage{amsopn}

\newtheorem{lem}{Lemma}[section]
\newtheorem{thm}{Theorem}[section]
\newtheorem{rem}{Remark}[section]

\newtheorem{ass}{Assumption}[section]

\newsiamremark{remark}{Remark}
\newsiamremark{assumption}{Assumption}
\newsiamremark{hypothesis}{Hypothesis}
\crefname{hypothesis}{Hypothesis}{Hypotheses}
\newsiamthm{claim}{Claim}

\headers{High-Dimensional Nonlinear Parabolic PDEs}{S. Fang, C. Sheng, B. Su, and  T. Zhou}

\title{A derivative-free localized stochastic method for very high dimensional semilinear parabolic PDEs\thanks{Submitted to the editors DATE.}}

\ifpdf
\hypersetup{
  pdftitle={Stochastic Method for Solving High-Dimensional Semilinear PDEs},
  pdfauthor={A, B, and  C}
}
\fi

\author{Shuixin Fang\thanks{Institute of Computational Mathematics and Scientific/Engineering Computing, Academy of Mathematics and Systems Science, Chinese Academy of Sciences, Beijing, 100190, P. R. China.
  (\email{sxfang@amss.ac.cn}). }
\and Changtao Sheng\thanks{School of Mathematics, Shanghai University of Finance and Economics, Shanghai 200433, China.
  (\email{ctsheng@sufe.edu.cn}).}
\and Bihao Su\thanks{School of Mathematics and Statistics, Hainan University, Haikou 570100, China.
  (\email{bihaosu@hainanu.edu.cn}).} 
\and Tao Zhou\thanks{Institute of Computational Mathematics and Scientific/Engineering Computing, Academy of Mathematics and Systems Science, Chinese Academy of Sciences, Beijing, 100190, P. R. China.
  (\email{tzhou@lsec.cc.ac.cn}). }}

\begin{document}

\maketitle

\begin{abstract}
We develop a mesh-free, derivative-free, matrix-free, and highly parallel localized stochastic method for high-dimensional semilinear parabolic PDEs. The efficiency of the proposed method is built upon four essential components:
(i) a martingale formulation of the forward backward stochastic differential equation (FBSDE);
(ii) a small scale stochastic particle method for local linear regression (LLR);
(iii) a decoupling strategy with a matrix-free solver for the weighted least-squares system used to compute $\nabla u$;
(iv) a Newton iteration for solving the univariate nonlinear system in $u$.
Unlike traditional deterministic methods that rely on global information, this localized computational scheme not only provides explicit pointwise evaluations of $u$ and $\nabla u$ but, more importantly, is naturally suited for parallelization across particles. In addition, the algorithm avoids the need for spatial meshes and global basis functions required by classical deterministic approaches, as well as the derivative-dependent and lengthy training procedures often encountered in machine learning. More importantly, we rigorously analyze the error bound of the proposed scheme, which is fully explicit in both the particle number $M$ and the time step size $\Delta t$. Numerical results conducted for problem dimensions ranging from $d=100$ to $d=10000$ consistently verify the efficiency and accuracy of the proposed method. Remarkably, all computations are carried out efficiently on a standard personal computer, without requiring any specialized hardware. These results confirm that the proposed method is built upon a principled design that not only extends the practically solvable range of ultra-high-dimensional PDEs but also maintains rigorous error control and ease of implementation. 

\end{abstract}
\begin{keywords}
High-dimensional PDEs, FBSDEs, Local linear regression, Stochastic particle methods, Error analysis, Parallel computing
\end{keywords}

\begin{MSCcodes}
65C30, 65M75, 60H30, 65N15, 68W10
\end{MSCcodes}

\maketitle

\section{Introduction}
\label{intro}
Partial differential equations (PDEs) in high dimensions constitute a fundamental modeling tool across diverse scientific and engineering disciplines, including quantitative finance, statistical physics, modern control, and learning systems. Typical examples comprise the Schrödinger equation in quantum many-body systems, the Black--Scholes equation in financial mathematics, and Hamilton--Jacobi--Bellman equations (HJB) in control and reinforcement learning \cite{HanJentzenE2018}. Despite their central role, the numerical treatment of such PDEs faces the notorious \emph{curse of dimensionality} (CoD), where the computational cost grows exponentially with the dimension. Classical deterministic discretization methods based on meshes or global bases, such as finite differences, finite elements, and spectral methods, quickly become infeasible once the dimension exceeds a moderate scale. Sparse grids markedly reduce degrees of freedom versus tensor-product meshes and remain effective up to about $d\approx10$ for smooth, mildly anisotropic solutions \cite{BungartzGriebel2004,Smolyak1963,Shen2010}. As $d$ and anisotropy increase, accuracy and conditioning degrade, and complexity grows exponentially in $d$, which restricts practical use. Thus, deterministic approaches remain fundamentally constrained by CoD when facing genuinely high-dimensional settings. 
 
Deep learning has established itself as a powerful tool for solving PDEs, and in recent years it has demonstrated notable strength in representing high-dimensional functions and mitigating the CoD, thereby emerging as a leading approach for high-dimensional PDEs. Existing methods can be broadly divided into two categories: i). direct learning; and ii). stochastic differential equations (SDE) based learning. Representative direct learning include Physics-Informed Neural Networks (PINN)\cite{RaissiPerdikaris2019}, the Deep Galerkin Method \cite{SirignanoSpiliopoulos2018}, the Deep Ritz method \cite{EYu2018}, and Weak Adversarial Networks  \cite{Zang2020}. In these methods, losses are computed at randomly sampled points, enabling efficient parallelization, while automatic differentiation for PDE derivatives remains challenging in very high dimensions, especially for $d \times d$ Hessians. To mitigate this issue, a stochastic-dimension gradient-descent variant of PINNs has been proposed \cite{Hu2024} and shows strong potential for ultra-high-dimensional PDEs.

In contrast to direct learning, SDE-based learning recast the problem as a backward stochastic differential equation (BSDE), which makes them inherently \emph{derivative-free}. In pioneering work, Han et al. \cite{EHan2017,HanJentzenE2018} introduced a deep BSDE framework that parameterizes the solution with neural networks and enforces the equations via residual minimization, solving PDEs in up to 100 dimensions. Related approaches include Deep Splitting and Deep Galerkin, and others (see, e.g.,\cite{BeckBecker2021,HurePhamWarin2020,Saporito2021,Zhang2022,Lu2024,Frey2025} and the references therein). Recently, Cai et al. \cite{CaiFangZhou2025,CaiFangZhou2024} introduced SOC-MartNet, a martingale-inspired architecture to solve HJB equations without explicit controls, and extended it to ultra-high-dimensional quasilinear parabolic equations, where it demonstrated strong performance on large-scale benchmarks. They later proposed a deep random difference method to reduce variance and improve stability \cite{CaiFangZhou2025Deep}. Despite these advances, several challenges remain: limited stability of the optimization procedure, pronounced sensitivity to hyperparameters, and a lack of rigorous a priori error estimates.

Similar to deep learning, stochastic methods constitute another class of numerical approaches that effectively mitigate the CoD and are widely applied across numerous scientific and engineering fields (see, e.g., \cite{Lord2014,Shao2020,Horton2020,Lei2025}). Unlike the black-box nature of deep learning, stochastic methods operate in a more transparent framework, which makes them suitable for error analysis. Their core is a probabilistic representation: the Feynman Kac formula for linear/nonlinear problems and FBSDEs for nonlinear problems, which eliminates explicit derivatives and replaces spatial meshes with Monte Carlo samples and conditional expectations (cf. \cite{PardouxPeng1990,MaYong1999,KloedenPlaten1992}). Recent advances, such as walk-on-spheres, show clear advantages for anomalous diffusion and other nonlocal effects (see, e.g., \cite{ShengSuXu2023,ShengSuXu2024}), because jump processes accelerate stochastic simulation relative to Brownian motion. As a result, for nonlocal problems with $d\ge 3$, stochastic methods are often more efficient than deterministic approaches. Nonetheless, their strengths lie primarily in high-dimensional linear cases, whereas nonlinear problems remain a substantial challenge (cf.~\cite{Yang2023}).

Extensive efforts have been made to confront the difficulties introduced by nonlinearities in stochastic algorithms. For example, probabilistic representations based on labeled branching diffusions with Malliavin automatic-differentiation weights absorb nonlinearity into branching, handle the $\nabla u$ term, and yield a Monte Carlo–ready random-variable representation (cf.~\cite{Henry-Labordere2019}). However, longer horizons or stiff dynamics cause rapid variance growth unless control variates and other variance-reduction techniques are used \cite{FahimTouziWarin2011,Henry-Labordere2019}. Hence, for nonlinear PDEs, probabilistic Monte Carlo methods based on BSDEs are more commonly used. These methods pair path simulation with regression-based estimators of conditional expectations, thereby avoiding spatial meshes, and proceed with a backward scheme to approximate $E_k[\cdot]$ via various regression methods (cf. \cite{GobetLemorWarin2005,BouchardTouzi2004,Zhang2004,BenderDenk2007}). In this way they retain the dimension-agnostic sampling of Monte Carlo and the clean measurability structure induced by filtrations. Nevertheless, accuracy and efficiency remain constrained by the bias--variance trade-off, the expressiveness and conditioning of the approximation spaces, and the distribution of samples in high-dimensional neighborhoods.

In spite of these advances, key gaps remain in stochastic methods for high-dimensional semilinear PDEs:
(i) the absence of a mesh-free, fully parallel solver capable of mitigating the CoD and providing dimension-independent comprehensive error analysis;
(ii) the lack of efficient and robust strategies to reconstruct $u$ and $\nabla u$ from particle ensembles.
The aim of this paper is to develop a mesh-free, derivative-free, matrix-free, and highly parallel localized stochastic method for high-dimensional semilinear PDEs, and to provide a rigorous error analysis.
The novel contributions of this article to the construction and analysis of stochastic method for semilinear parabolic equation in very high dimensions include the following several aspects.
\begin{itemize}
\item {\bf Derivative-free and pointwise local solver}: By casting the semilinear equation (see \eqref{mainprob}) as a corresponding FBSDE and using a martingale formulation, we rigorously link PDEs to stochastic processes. This connection underpins two key advantages of our stochastic algorithm over traditional deterministic methods. First, it entirely eliminates derivative computations, including gradients and Hessians, which are prohibitively expensive in ultra high dimensions, even for deep neural networks. Second, it transforms global discretization into a genuinely local solver, enabling scalable, pointwise computations that are both simple and naturally parallel. 

\item {\bf Small-scale local particle method}:
We employ Gaussian weights to enhance particle discriminability and select all particles in the ensemble, thereby eliminating the radius tuning required in conventional LLR. This contrasts sharply with $k$-nearest neighbors (kNN), which in high dimensions tends to induce inflated radii and distance concentration (cf.~\cite{Aggarwal2001}). In our analysis, the particle number $M$ enters only through an exponentially suppressed bad-event probability ${\rm e}^{-cM}$ (cf.~\eqref{eq:grad-bv}), so a moderate $M$ suffices, and the numerical evidence in Section~\ref{numexp} confirms that $M\approx100$ already attains accurate results.

\item {\bf Decoupled scheme for $u$ and $\nabla u$ and a matrix-free solver}: 
Unlike existing work \cite{GobetLemorWarin2005}, which relies on Picard iterations to solve the coupled nonlinear system involving $u$ and $\nabla u$ and often causes a dramatic increase in computational cost in high dimensions, we adopt a decoupling strategy. Specifically, we first approximate the gradient $\nabla u$ via LLR by solving a least-squares problem. The associated $(d+1)\times(d+1)$ linear system is solved in a matrix-free manner, so the storage requirement is $\mathcal{O}(d)$ and the per-time-step cost is only $\mathcal{O}(Md)$, where $M$ denotes the number of particles. Once $\nabla u$ is obtained, the remaining univariate nonlinear equation in $u$ can be solved straightforwardly. This design enables efficient handling of problems in very high dimensions.

\item {\bf Analyzable computational framework}: 
Built on an interpretable computational framework, our algorithm admits a rigorous error bound of $\mathcal{O}(\Delta t)+\mathcal{O}(\Delta t\,e^{-cM})$ (cf.\,Theorem \ref{ITE}), where $M$ denotes the number of particles and $\Delta t$ the time-step size. This result demonstrates first-order temporal accuracy and requires only the selection of an appropriate number of particles, and these theoretical findings are fully corroborated by numerical experiments. 

\end{itemize}

The rest of the paper is organized as follows. In Section~\ref{mainalgo}, we introduce the standing assumptions and provide a detailed description of the complete stochastic algorithm. Section~\ref{erranalysis} presents the necessary preparations for the theoretical analysis and then establishes rigorous convergence results. The numerical aspects are discussed in Section~\ref{numexp}, where extensive high-dimensional numerical experiments are evaluated to demonstrate the accuracy, efficiency, and robustness of the proposed method. We conclude in Section~\ref{conclusion} with final remarks and an outlook on future research directions.

\section{Main algorithm}\label{mainalgo}
In this section, we first present the problem together with the associated FBSDEs, and then provide a detailed description of the proposed stochastic algorithm. The procedure begins with employing the martingale formulation for time discretization. Subsequently, a local stochastic particle method combined with a localized reconstruction strategy is introduced, and a Newton iteration is finally applied to resolve the resulting pointwise nonlinear systems.

\subsection{Problem setting}\label{problemsetting}
Consider the following semilinear parabolic PDE defined on $[0,T]\times\mathbb{R}^d$:  
\begin{equation}
\label{mainprob}
\begin{cases}
(\partial_{t} + \mathcal{L}) u(t,\bx)
+ f\bigl(t,\bx,u(t,\bx),\sigma^\top\nabla u(t,\bx)\bigr) = 0, 
& (t,\bx)\in[0,T)\times\mathbb{R}^d,\\[1ex]
u(T,\bx) = g(\bx), & \bx\in\mathbb{R}^d,
\end{cases}
\end{equation}
where $u:[0,T]\times\mathbb{R}^d \to \mathbb{R}$ is the unknown scalar function,
and $\mathcal{L}$ denotes the infinitesimal generator of the underlying Itô (or Lévy-type) process,  
\begin{equation*}
\mathcal{L}u(t,\bx) 
= \tfrac{1}{2} \text{Tr} \big(\sigma(t,\bx)\sigma(t,\bx)^{\top} \text{Hess}_{\bx}u(t,\bx)\big) 
+ \langle \mu(t,\bx), \nabla u(t,\bx)\rangle .
\end{equation*}
Here $\nabla u$ and $\text{Hess}_{\bx}u$ denote the gradient and the Hessian of $u$ with respect to $\bx$, 
$\sigma:[0,T]\times\mathbb{R}^d\to\mathbb{R}^{d\times d}$ is the matrix-valued diffusion coefficient, 
$\mu:[0,T]\times\mathbb{R}^d\to\mathbb{R}^d$ is the vector-valued drift coefficient, 
$f:[0,T]\times\mathbb{R}^d\times\mathbb{R}\times\mathbb{R}^d\to\mathbb{R}$ is a nonlinear source term, 
and $g:\mathbb{R}^d\to\mathbb{R}$ prescribes the terminal condition.  
In particular, we are often interested in evaluating the solution at the initial time $t=0$ and spatial location $\bx=\xi$ for some $\xi\in\mathbb{R}^d$.

In the semilinear case, $u$ admits an FBSDE characterization, whereas if the nonlinearity depends explicitly on $\nabla^2 u$, one may employ second-order BSDEs (cf.~\cite{Cheridito2007}) or adopt local surrogate models for the Hessian. In this work, we focus on a very high-dimensional setting $d \gg 1$ where the nonlinearity $f$ 
involves only gradient terms. To this end, we introduce the stochastic processes
\begin{equation}\label{def_YZ}
Y_t = u(t,X_t),    \;\;\;
Z_t = \sigma^\top(t,X_t) \nabla u(t,X_t),
\end{equation}
where the forward process $\{X_t\}_{t \ge 0}$ solves the following SDE
\begin{equation}\label{dXt}
\d X_t= \mu(t,X_t) \d t + \sigma(t,X_t) \d W_t,   \;\;\;  X_0 = \bx,
\end{equation}
and $W_t$ is a $d$-dimensional Brownian motion.  
It then follows that \eqref{mainprob} is equivalent to the coupled forward--backward system
\begin{equation}\label{FBsyst}
\begin{cases}
\d X_t = \mu(t,X_t) \d t + \sigma(t,X_t) \d W_t, & X_0 = \bx, \\[6pt]
\d Y_t = - f\bigl(t,X_t,Y_t,Z_t\bigr) \d t + Z_t^\top \d W_t, & Y_T = g(X_T).
\end{cases}
\end{equation}

This FBSDE formulation provides the foundation for probabilistic algorithm, 
as it allows the original high-dimensional PDE to be reformulated as a system of stochastic equations 
that can be solved by various discretization techniques for FBSDEs, 
including more recent works based on deep neural networks (see e.g.,\cite{CaiFangZhou2025,CaiFangZhou2025Deep,CaiFangZhou2024,HanJentzenE2018}). 
\begin{ass}[{\bf Global Lipschitz and linear growth}]\label{ass:HLG}
Let $\mu:[0,T]\times\R^d\to\R^d$, $\sigma:[0,T]\times\R^d\to\R^{d\times d}$,
$f:[0,T]\times\R^d\times\R\times\R^{d}\to\R$, and $g:\R^d\to\R$.
There exist constants $L,C>0$ such that for all $t\in[0,T]$, $\bx,\bx^\prime\in\R^d$,
$\by,\by^\prime\in\R$, and $\bz,\bz^\prime\in\R^{d}$, the following hold
\begin{enumerate}
\item{Global Lipschitz}:
\begin{equation*}\begin{split}
&\|\mu(t,\bx)-\mu(t,\bx^\prime)\| +\|\sigma(t,\bx)-\sigma(t,\bx^\prime)\| \le L\,\|\bx-\bx^\prime\|,\\
&|f(t,\bx,\by,\bz)-f(t,\bx^\prime,\by^\prime,\bz^\prime)| \le L\bigl(\|\bx-\bx^\prime\|+|\by-\by^\prime|+\|\bz-\bz^\prime\|\bigr),\\
&|g(\bx)-g(\bx^\prime)| \le L\,\|\bx-\bx^\prime\|;
\end{split}\end{equation*}
\item{Linear growth}: 
\begin{equation*}
\|\mu(t,\bx)\| + \|\sigma(t,\bx)\| + |f(t,\bx,\by,\bz)| + |g(\bx)|\le C\bigl(1 + \|\bx\| + |\by| + \|\bz\|\bigr).
\end{equation*}
\end{enumerate}
Here $\|\cdot\|$ denotes the Euclidean norm in the relevant space.
\end{ass}

To ensure the well-posedness of this FBSDE formulation, we recall below a classical result 
under standard Lipschitz and growth conditions (cf.~\cite{Higham2005}).
\begin{lem} 
\label{Lipcond} 
Suppose {\rm Assumption~\ref{ass:HLG}} holds. We further assume that $\sigma\sigma^\top$ is uniformly nondegenerate, i.e., 
there exists $\lambda>0$ such that
\begin{equation*}
\xi^\top \big(\sigma(t,\bx)\sigma(t,\bx)^\top\big)\xi 
\ge \lambda \|\xi\|^2, 
\quad \forall \xi\in\mathbb{R}^d, \;\; (t,\bx)\in[0,T]\times\mathbb{R}^d,
\end{equation*}
then the FBSDE admits a unique adapted solution
$(X,Y,Z)\in \mathcal{S}^2(\mathbb{R}^d)\times \mathcal{S}^2(\mathbb{R})\times \mathcal{H}^2(\mathbb{R}^d),$
where $\mathcal{S}^2$ denotes the space of square‐integrable continuous adapted processes, and 
$\mathcal{H}^2$ denotes the space of square‐integrable predictable processes.
\end{lem}

\subsection{Time Discretization based on Martingale formulation}
We construct a uniform time grid on the interval $[0,T]$ by dividing it into $N$ subintervals of equal length $\Delta t = T/N$, and denote the discrete time nodes by $t_k = k\Delta t$ for $k=0,1,\dots,N$.  
Starting from the backward stochastic differential equation \eqref{FBsyst}:
$$\d Y_t = -f(t, X_t, Y_t, Z_t) \d t + Z_t^\top  \d W_t,$$
together with the representation $Z_t = \sigma^\top(t, X_t)\nabla u(t, X_t)$,  
we integrate both sides over the subinterval $[t_k, t_{k+1}]$ to obtain
$$Y_{k+1} - Y_k
= -\int_{t_k}^{t_{k+1}} f(s, X_s, Y_s, Z_s) \d s
+ \int_{t_k}^{t_{k+1}} Z_s^\top  \d W_s,$$
here $Y_{k} = Y_{t_k} = u(t_{k},X_{t_{k}})$. 
Because the dynamics evolve backward in time \cite{ZhaoChenPeng2006}, this relation can be rearranged as
$$Y_k = Y_{k+1} + \int_{t_k}^{t_{k+1}} f(s, X_s, Y_s, Z_s) \d s 
- \int_{t_k}^{t_{k+1}} Z_s^\top \d W_s.$$
Noting that $Y_{k}$ is $\mathcal{F}_{t_k}$-measurable, we introduce the conditional expectation with respect to the filtration $\mathcal{F}_{t_k}$, namely $\E_k[\cdot] := \E[\cdot \mid \mathcal{F}_{t_k}]$. Using the fact that the Itô integral has zero conditional expectation, i.e.,
$\E_k\big[\int_{t_k}^{t_{k+1}} Z_s^\top   dW_s\big] = 0,$
we obtain the following recursion:
\begin{equation}
\label{Ytk}
Y_k
= \E_k\Bigl[Y_{k+1} + \int_{t_k}^{t_{k+1}} f\bigl(s,X_s,Y_s,Z_s\bigr)  \d s\Bigr].
\end{equation}
To discretize the integrals over $[t_k, t_{k+1}]$, we apply a first-order Euler--Maruyama approximation by freezing the coefficients at $t_k$, which gives
\begin{equation*}
\int_{t_k}^{t_{k+1}} f(s,X_s,Y_s,Z_s) \d s
 \approx  f\bigl(t_k,X_{k},Y_k,Z_k\bigr)  \Delta t,
\quad
\int_{t_k}^{t_{k+1}} Z_s^\top   \d W_s
 \approx  Z_k^\top \Delta W_k,
\end{equation*}
where $\Delta W_k := W_{t_{k+1}} - W_{t_k}$ denotes the Brownian increment and $Z_{k} = Z_{t_{k}}$. 
Substituting these approximations into the conditional expectation relation \eqref{Ytk} and denoting the numerical solution by $\{\Y_k\}_{k=0}^N$ yield the semi-discrete backward scheme
\begin{equation}
\label{Ytkdiscrete}
\Y_{k}=\E_k\Bigl[\Y_{k+1} + f\bigl(t_k,X_{k},\Y_{k},\Z_k\bigr)\Delta t \Bigr],\;\;\; 0\leq k\leq N.
\end{equation}
With this foundation, we next focus on solving univariate nonlinear systems involving expectation operators using a local stochastic particle methods.

\subsection{Stochastic particle method}
This subsection develops a stochastic particle approximation of the conditional expectation in \eqref{Ytkdiscrete}. For the $m$-th particle at time $t_k$, the conditional expectation $\E_k[\cdot]$ is taken with respect to the filtration $\mathcal{F}_{t_k}$ generated by the ensemble of particle positions $\mathcal{S}=\{X_k^{1},\ldots,X_k^{M}\}$. More precisely, since $\Y_k$ and $\Z_k$ are $\mathcal{F}_{t_k}$-measurable, we approximate, for each particle $X_k^{m}$, the conditional expectation in \eqref{Ytkdiscrete} by the empirical average over all particles at time $t_{k+1}$. In practice, one may certainly select a small subset of the nearest particles from $\mathcal{S}$ to perform the regression instead of using all particles. However, since our algorithm uses only a small number of particles (typically $M\le 100$) and the computation for each particle is fully parallelizable, we using all particles for the regression for notational simplicity.

To this end, we first simulate $M$ independent particle trajectories $\{X^{j}_k\}_{j=1}^M$ by the Euler--Maruyama discretization of the forward SDE, and denote the numerical solution of $j$-th particle at time $t_k$ by $\X_k^j$:
\begin{equation}\label{Xk}
\X^{j}_{k+1}= \X^{j}_k + \mu\bigl(t_k,\X^{j}_k\bigr)\Delta t 
+ \sigma\bigl(t_k,\X^{j}_k\bigr)\Delta W_k^{j},  
\quad j=1,2,\cdots,M,
\end{equation}
where $\Delta W_k^{j} \sim \mathcal{N}(0, \Delta t I)$ are independent Brownian increments.

Since both $\Y_k$ and $\Z_k$ are $\mathcal{F}_{t_k}$--measurable, the solution of discrete scheme \eqref{Ytkdiscrete} can, for each particle $\X_k^{m}$, be approximated as
\begin{equation}
\label{discreteYkm}
\begin{split}
\Y_k^{m} 
&= \E_k\Big[\Y_{k+1}    \big| \X^m_k\Big] +  f\bigl(t_k, \X_k^{m}, \Y_k^{m}, \Z_k^{m}\bigr) \Delta t  \\
&\approx \frac{1}{M} \sum_{j=1}^{M} \Y_{k+1}^{j} +f\bigl(t_k, \X_k^{m}, \Y_k^{m}, \Z_k^{m}\bigr) \Delta t,\;\;\;        1\leq m\leq M,
\end{split}
\end{equation}
where the conditional expectation is estimated by a local averaging procedure over those stochastic particles $\{\X_k^{j}\}_{j=1}^{M}$ whose positions fall within a neighborhood of $\X_k^{m}$. By recursively applying this procedure \eqref{discreteYkm} backward in time from $k = N-1$ to $k = 0$, the approximation of the solution at $t=0$ is given by the particle average
\begin{equation*}
\Y_0 = \frac{1}{M} \sum_{m=1}^M \Y_0^{m}.
\end{equation*}

\begin{rem}{
A salient feature of our method is its sample efficiency: accuracy is attainable with few particles, often with only $100$. This accords with {\rm Theorem~\ref{ITE}}, where the error bound contains $\Delta t,e^{-c M}$. An appropriate choice of $M$ ensures first-order accuracy. However, for challenging problems, more particles may be needed to maintain accuracy. In such cases, with a suitable $\varepsilon_k$, the method can be viewed as a \emph{random batch method} {\rm (cf.~\cite{jin2020})}, where reconstruction at each point uses only a fixed, small set of nearest neighbors, keeping the overall computational cost $O(M)$.} 
\end{rem}

\subsection{Computation of $\{Z_k^m\}_{m=1}^M$ via Local Linear Regression}
The principal difficulty in efficiently solving \eqref{discreteYkm} arises from its structure as a coupled $(d+1)$-dimensional nonlinear system in the variables $\Y_k^m$ and $\Z_k^{m}$. The approach proposed in \cite{GobetLemorWarin2005} relies on applying Picard iterations directly to this $(d+1)$-dimensional system, in conjunction with indicator functions on hypercubes for function reconstruction. While effective in low dimensions, this strategy becomes computationally prohibitive as the dimension increases. To overcome this challenge, we adopt a decoupling strategy: the $d$-dimensional component $\Z_k^{m}$ is first approximated, after which the resulting univariate nonlinear system in $\Y_k^m$ is solved.
Therefore, the objective of this subsection is to estimate $Z^{m}_k = \sigma^\top(t_k, X_k^{m})   \nabla u(t_k, X_k^{m})$
by computing the spatial gradient $\nabla u(t_k,X_k^{m})$. To this end, we approximate the function $u(t_{k+1}, \cdot)$ in a neighborhood of $\X_k^{m}$ via a first-order Taylor expansion:
\begin{equation}\label{taylorexp}
u(t_{k+1}, \cdot\,) \approx u(t_k, \X_k^{m}) + \partial_t u(t_k, \X_k^{m})   \Delta t + \nabla u(t_k, \X_k^{m})^\top (\;\cdot - \X_k^{m}).
\end{equation}
where $\cdot$ denotes the spatial variable and the time is fixed at $t_{k+1}$. 
To approximate the gradient, we employ a local linear regression using all particles $\{\X_k^{j}\}_{j=1}^{M}$ within the $\varepsilon_k$-neighborhood of $\X_k^{m}$.
It is important to note that, due to the backward-in-time nature of the algorithm, the values $u(t_{k+1}, \X^{j}_{k+1})$ have already been computed in the previous step, whereas the values at $t_k$ are yet to be updated.

We now present the detailed procedure for estimating the gradient $\nabla u(t_k,\X_k^{m})$ at time $t_k$. Since $\X^{j}_{k+1} =\X^{j}_k + \Delta X^{j}$ and $\Delta X^{j}$ is known, the value $u(t_{k+1}, \X^{j}_{k+1})$ can be regarded as a function of $\X^{j}_k$. In the fitting process, we directly perform a linear regression in the $\X_k$-space using the pairs $\{(\X_k^{j}, \Y_{k+1}^{j})\}_{j=1}^{M}$. 
To this end, we adopt a local linear approximation centered at the anchor point $\X_k^{m}$. 
For notational convenience, we set  
$$\alpha := u(t_k,\X_k^{m}) + \partial_t u(t_k,\X_k^{m}) \Delta t\in \R,$$  
and define 
$$\alpha_{\bx} := \nabla u(t_k,\X_k^{m}) = \big(\partial_{x_1} u(t_k,\X_k^{m}), \ldots, \partial_{x_d} u(t_k,\X_k^{m})\big)^\top \in \R^d.$$
Then, for each $\X_k^{m}$, the unknown coefficients $\boldsymbol{\alpha}:=(\alpha;\alpha_{\bx})\in\R^{d+1}$ in \eqref{taylorexp} are obtained by minimizing the weighted least-squares functional:
\begin{equation}\label{leastsq}
J(\boldsymbol{\alpha})=\sum_{j =1}^{M}
w_j \Big(\Y_{k+1}^{j}- \alpha- {\alpha}_{\bx}^\top \big(\X_k^{j} - \X_k^{m}\big)\Big)^2, \quad 1\leq m\leq M,
\end{equation}
where $w_j$ denotes the weight assigned to each neighbor. When the particle distribution is non-uniform, weighted least squares can significantly reduce estimation variance.

 We compute the coefficient vector $\boldsymbol{\alpha} \in\mathbb{R}^{d+1}$ by minimizing the weighted sum of squared residuals, where the weights $w_j$ are determined based on the proximity of each $\X_k^{j}$ to the anchor point $\X_k^{m}$. Specifically, we define
\begin{equation}
\label{Dj}
D_j := \X_k^{j} - \X_k^{m}\in \R^d, \;\;\;
w_j := \frac{ K\big( \frac{ \| D_j \| }{\varepsilon_k} \big) }{ \sum_{i =1}^{M} K\big( \frac{ \| D_i \| }{\varepsilon_k} \big) },
\end{equation}
where $K$ is a given kernel function (e.g., the Gaussian kernel), and $\varepsilon_k > 0$ represents the maximum distance between point $ \X_k^{j}$ and $ \X_k^{m}$.
As a result, the weighted least squares objective \eqref{leastsq} reads
\begin{equation}
\label{WLS_objective}
J(\boldsymbol{\alpha}) := \sum_{j=1}^{M}
w_j \left( \Y_{k+1}^{j} - \alpha -  \mathbf{\alpha}_{\bx}^\top D_j \right)^2.
\end{equation}
 To minimize the objective functional \eqref{WLS_objective} with respect to $\boldsymbol{\alpha} \in\mathbb{R}^{d+1}$, we set its gradient to zero, leading to the normal equations:
 \begin{equation}\label{least2}
\begin{split}
&\frac{\partial J}{\partial \alpha}
=-2\sum_{j=1}^{M} w_j   (\Y_{k+1}^{j} - \alpha   - \mathbf{\alpha}_{\bx}^\top D_j)
=0, 
\\&\frac{\partial J}{\partial \alpha_{\bx}}
=-2\sum_{j =1}^{M} w_j   (\Y_{k+1}^{j} - \alpha   - \mathbf{\alpha}_{\bx}^\top D_j)\cdot D_{j}
=0.
\end{split}
\end{equation}
We now define the design matrix, response vector, and weight matrix as
\begin{equation}
\label{AY}
\boldsymbol{D} =\begin{pmatrix}
1   & D_{1}^{\top}\\
1  & D_{2}^{\top}\\
\vdots   & \vdots\\
1   & D_{M}^{\top}\end{pmatrix}\in \R^{M\times(d+1)} ,
\;\; \boldsymbol{Y} =\begin{pmatrix}
\Y_{k+1}^{1}\\
\Y_{k+1}^{2}\\
 \vdots\\
\Y_{k+1}^{M}\end{pmatrix}\in\R^{M},
\;\; \boldsymbol{W} = \begin{pmatrix}
w_1 &&&\\
&w_2 &&\\
 &&\ddots&\\
&&&w_{M}\end{pmatrix}.
\end{equation}
With these definitions, the system \eqref{least2} can be rewritten compactly as
\begin{equation}
\label{normal-eq}
(\boldsymbol{D}^\top\boldsymbol{W}\boldsymbol{D})\boldsymbol{\alpha}= \boldsymbol{D}^\top \boldsymbol{W} \boldsymbol{Y}.
\end{equation}
Under the condition $\sum_{j} w_{j} D_{j} = 0$, the weighted least squares problem \eqref{normal-eq} admits a unique solution.
In actual computation, we adopt a \emph{matrix-free} strategy: iterative Krylov solvers such as LSQR or preconditioned conjugate gradient (PCG) are applied, where only matrix--vector products with $\boldsymbol{D}$ and $\boldsymbol{D}^\top$ are required.  
This approach reduces the cost to $\mathcal{O}(M d)$ per time step and avoids storing $\boldsymbol{D}$ or explicitly forming $\boldsymbol{D}^\top\boldsymbol{W}\boldsymbol{D}$.
Specifically, for any given vector $\bs{\alpha}=(\alpha,\alpha_{\bx})^\top\in\mathbb{R}^{d+1}$, 
the matrix--vector products in the left side of \eqref{normal-eq} are computed in two steps as follows
\begin{description}
  \item[Step 1] {Forward product with weights ($\boldsymbol{W}\boldsymbol{D}\bs{\alpha}$)}: for each $j=1,\dots,M$,
$$ \bs{\beta}_j:= (\boldsymbol{W}\boldsymbol{D}\bs{\alpha})_j = w_j\big(\alpha + D_j^\top \alpha_{\bx}\big).$$

  \item[Step 2] {Transpose product ($\boldsymbol{D}^\top \bs{\beta}$)}: for $\bs{\beta}\in\mathbb{R}^{M}$,
 $$(\boldsymbol{D}^\top \bs{\beta})_0 = \sum_{j=1}^{M} \bs{\beta}_j, 
  \quad  (\boldsymbol{D}^\top \bs{\beta})_{1:d} = \sum_{j=1}^{M} \bs{\beta}_j D_j.$$
\end{description}

This matrix-free scheme for \eqref{normal-eq} achieves linear complexity in both $M$ and $d$ per time step and is thus particularly suitable for very high-dimensional problems.
The vector $\boldsymbol{\alpha}$ is of interest only through its last $d$ components, which correspond to the spatial gradient $\nabla u(t_k,\X_k^{m})$. The term $\Z^{m}_k$ is then computed as $\Z^{m}_k = \sigma^\top(t_k, \X_k^{m}) \nabla u(t_k, \X_k^{m}),$ 
whereas the first component of $\boldsymbol{\alpha}$, denoted by $\alpha$, is irrelevant to this computation and is therefore discarded.

\begin{rem}{
In high-dimensional settings, it often occurs that the number of particles $M\ll d$, which renders the normal equations underdetermined or severely ill-conditioned. To address this issue in practical computations, we adopt ridge regression (also known as Tikhonov regularization). Specifically, instead of solving the weighted least-squares problem in its original form, we minimize the penalized functional
$$
J_\lambda(\boldsymbol{\alpha}) = \sum_j w_j \big(\Y_{k+1}^j - \alpha - \alpha_{\bx}^\top D_j \big)^2 + \lambda \|\boldsymbol{\alpha}\|^2, 
\quad \lambda > 0,
$$
which leads to the regularized solution
$$
\boldsymbol{\alpha} = (\boldsymbol{D}^\top W \boldsymbol{D} + \lambda I)^{-1}\boldsymbol{D}^\top \boldsymbol{W}   \boldsymbol{Y}.
$$
The additional penalty term $\lambda \|\boldsymbol{\alpha}\|^2$ guarantees the invertibility of the system matrix and improves numerical stability, while only introducing a mild bias. This regularization is particularly effective when $M$ is small relative to $d$, as it balances variance reduction and stability in the estimation of $\Z_k^m$.}
\end{rem}

\begin{rem}  {
Despite the concentration of Euclidean distances as the dimension increases, the LLR step in the proposed method remains efficient, owing to a kernel-based prioritization by relative distance. For the Gaussian kernel $K(u)=e^{-u^2}$,
$$
\frac{w_j}{w_i}=\exp\!\Big(-\frac{\|D_j\|^2-\|D_i\|^2}{\varepsilon_k^2}\Big).
$$
Even if the absolute distances $\|D_j\|$ concentrate, the relative gap $\left|\|D_j\| - \|D_i\|\right|$ still provides discriminative weights that favor nearer neighbors. In addition, since the number of particles $M$ is typically small, one can readily identify enough neighbors in each local region. Consequently, through these complementary mechanisms, the FBSDE–LLR framework substantially improves the reliability of neighborhood selection and effectively overcomes the inherent limitations of classical LLR methods.}
   \end{rem}

\subsection{Computation of $\{\Y_k^m\}_{m=1}^M$ via Newton iteration}
With both $\X^{m}_k$ and $\Z^{m}_k$ specified, the nonlinear system \eqref{discreteYkm} reduces to a one-dimensional equation in $Y^{m}_k$, which is subsequently solved in the backward update
\begin{equation*}
\Y^{m}_k = \frac{1}{M}   \sum_{j}\Y^{j}_{k+1}+  f\bigl(t_k,\X^{m}_{k},\Y^{m}_k,\Z^{m}_k\bigr) \Delta t, \quad  1\leq m\leq M.
\end{equation*}
To this end, we define the following nonlinear function of $\Y^{m}_k$:
\begin{equation}
\label{nonlinearf}
F(\Y^{m}_{k})=\Y^{m}_{k} -\frac{1}{M}   \sum_{j=1}^M\Y^{j}_{k+1}+f\bigl(t_k,\X^{m}_{k},\Y^{m}_k,\Z^{m}_k\bigr)\Delta t,
\end{equation}
such that the desired solution $\Y^{m}_k$ corresponds to a root of $F$.
To solve the nonlinear equation \eqref{nonlinearf}, one may employ various numerical solvers. In this work, we adopt the Newton iteration method, which iteratively updates the solution via
\begin{equation}
\label{NRiter}
 \Y^{m,(n+1)}_k= \Y^{m,(n)}_k- \frac{F\bigl(\Y^{m,(n)}_k\bigr)}{F^\prime\bigl(\Y^{m,(n)}_k\bigr)},  \quad    n=0,1,2,\cdots, \quad 1\leq m\leq M,
\end{equation}
where $F^\prime$ denotes the derivative of $F$ with respect to $\Y^{m}_k$. For clarity, we summarize the complete algorithm as follows.

 \begin{algorithm}
\caption{FBSDE Solver with Local Linear Regression method for \eqref{mainprob}.}  \label{alg:Framwork}
\begin{algorithmic}[1]
        \REQUIRE  $T$: terminal time; $d$: spatial dimension; $M$: particle count; $N$: time step count; $\Delta t$: time step size; $\bx$: target point     \\
\FOR{$j=1:M$({\bf in parallel})}
\STATE Set the terminal condition $Y_{N}^{j} = g(X_{N}^{j})$;
\ENDFOR

\FOR{$k=1:N$}
\FOR{$j=1:M$({\bf Forward in parallel})}
\STATE {Simulate the trajectories of particles} $\X_{k}^{j}$ by \eqref{Xk};
\ENDFOR
\ENDFOR

\FOR{$k=N-1:0$}
\FOR{$m=1:M$({\bf Backward in parallel})} 
\STATE \hspace{3pt} Compute $\boldsymbol{\alpha}=(\alpha,\alpha_{\bx})^\top$ by a matrix-free solver applied to \eqref{normal-eq}.
\STATE \hspace{3pt} Compute $\nabla u \leftarrow \alpha_{\bx}$ and $\Z_k \leftarrow \sigma^\top \nabla u$; 
\STATE \hspace{3pt} Compute $\Y^m_k$ by using the Newton method \eqref{NRiter};  
\ENDFOR
\ENDFOR
 
\STATE Calculate the estimated value of $\Y_{0} = \frac1M\sum_{m=1}^{M}\Y_{0}^{m}$.
  \ENSURE The estimated value of the initial value $u(0,\bx)$;
\end{algorithmic}
\end{algorithm}

\section{Error estimates}\label{erranalysis}

In this section, we first analyze the various sources of error in the computation process and introduce several auxiliary lemmas that will be used in the final error analysis. These include the time discretization error of stochastic differential equations, stochastic matrix estimates associated with linear regression along random paths, and truncation errors from stochastic expansions. Finally, we present a rigorous error analysis tailored to the proposed algorithm.

We recall a classical result on the strong error of the Euler–Maruyama scheme \eqref{Xk} for $\X_k$ (see \cite[Theorem 10.2.2]{KloedenPlaten1992}), which determines its convergence order and forms the basis of our analysis. 
\begin{lem} {\rm \bf(Strong convergence of forward SDE for $\X_k$)}
Let $X_t$ be the solution of \eqref{dXt}, where the coefficients $\mu$ and $\sigma$ satisfy global Lipschitz continuity and linear growth conditions
\begin{equation*}
\begin{split}
 &\|\mu(t,\bx)-\mu(t,\bx')\| \leq L\|\bx-\bx'\|,     \|\sigma(t,\bx)-\sigma(t,\bx')\| \leq L\|\bx-\bx'\|,\\
&\|\mu(t,\bx)\| \leq K(1+\|\bx\|),\hspace{47pt} \|\sigma(t,\bx)\| \leq K(1+\|\bx\|).
\end{split}
\end{equation*}
The continuous-time Euler--Maruyama approximation is then defined for $t\in[t_k,t_{k+1})$ by
$$\widetilde{X}_t = X_{t_k} +  \mu(t_k,X_{t_k}) (t-t_k)
+ \sigma(t_k,X_{t_k})(W(t) - W(t_k)),$$
Clearly, for $t=t_k$ this reduces to the standard Euler--Maruyama scheme.  
Moreover, the scheme is known to achieve strong convergence of order $1/2$, in the sense that
\begin{equation}\label{Xbound}\max_{0\le t\le T} \E \left[\|X_{t}-\widetilde{X}_t\|^2\right] \le C \Delta t,\qquad\E\Big[ \sup_{0\le t\le T} \|X_t - \widetilde{X}_t\|^2 \Big] \le C \Delta t,\end{equation}
where $C>0$ is a constant depending only on $L,K,T$ and the initial location $X_0=\bx$.
\end{lem}

The following discrete Gronwall Lemma can be found in
\cite{Brunner2004}.
\begin{lem}
\label{le3.2} Assume that $\{k_j\}~(j\geq 0)$ is a given
non-negative sequence, and the sequence $\{\varepsilon_n\}$
satisfies $\varepsilon_0\leq\rho_0$ and
\begin{equation}\label{eq:3.2}
\varepsilon_n\leq
\rho_0+\sum^{n-1}_{j=0}q_j+\sum^{n-1}_{j=0}k_j\varepsilon_j,\quad
n\geq 1,
\end{equation}
with $\rho_0\geq0$, $q_j\geq 0~(j\geq 0)$. Then
\begin{equation}\label{eq:3.3}
\varepsilon_n\leq
\big(\rho_0+\sum^{n-1}_{j=0}q_j\big)\exp(\sum^{n-1}_{j=0}k_j),\quad
n\geq 1.
\end{equation}
\end{lem} 

We now analyze the time discretization error of the semi-discrete Euler Maruyama scheme for the Martingale formulation of backward SDE associated with $\Y_k$ in \eqref{Ytkdiscrete}.
\begin{lem}{\rm \bf (Discretization error for Euler scheme \eqref{Ytkdiscrete})}
\label{lem:DEES} 
If $f \in C^{1,2}$ and satisfies the Lipschitz condition \eqref{Lipcond}, then the local truncation error of semi-discrete backward scheme \eqref{Ytkdiscrete} is bounded by
\begin{equation}\label{c}
|\mathcal{E}_{k}| := \Big|\E_k \left[ \int_{t_k}^{t_{k+1}} f(s,X_s,Y_s,Z_s) \d s \right] - f(t_k,X_k,Y_k,Z_k)  \Delta t \Big| \leq C (\Delta t)^2,
\end{equation}
where $C$ is a positive constant independent of $\Delta t$.
\end{lem}
 
\begin{proof}
For ease of notation, set $F(t,\bx):=f\bigl(t,\bx,\,u(t,\bx),\,(\nabla_{\!x}u)(t,\bx)\,\sigma(t,\bx)\bigr)$ so that the discretization error \eqref{c} satisfies
\begin{equation*}\begin{split}
\mathcal{E}_{k}&= \E_k\!\left[\int_{t_k}^{t_{k+1}} f\bigl(s,X_s,Y_s,Z_s\bigr)\,\mathrm{d}s\right]
- f\bigl(t_k,X_k,Y_k,Z_k\bigr)\,\Delta t
\\&= \int_{t_k}^{t_{k+1}}\!\Bigl(\E_k\big[F(s,X_s)\big]-F(t_k,X_k)\Bigr)\,\mathrm{d}s.
\end{split}\end{equation*}
which implies
\begin{equation}\label{eq:Ek-abs}
|\mathcal{E}_{k}|  \le  \int_{t_k}^{t_{k+1}} \bigl|\E_k[F(s,X_s)]-F(t_k,X_k)\bigr| {\rm d} s.
\end{equation}
Applying Itô’s formula to $F(t,\bx)$ yields
\begin{equation}\label{eq:ItoF}
F(t,X_t)  =  F(t_k,X_k) + \int_{t_k}^{t} (\partial_t+\mathcal L)F(s,X_s)\,\d s
+ \int_{t_k}^{t} \nabla_x F(s,X_s) \sigma(s,X_s)\,\d W_s,
\end{equation}
where the generator $\mathcal L$ is the one defined in \eqref{mainprob}.
Taking conditional expectation and differentiating in $t$ yields
\begin{equation}\label{eq:cond-time-deriv}
\frac{\d}{\d t} \E_k \big[F(t,X_t)\big]  =  \E_k \big[(\partial_t+\mathcal L)F(t,X_t)\big],\qquad t\in[t_k,t_{k+1}].
\end{equation}
Therefore,
\[
\sup_{t\in[t_k,t_{k+1}]}\left|\frac{\d}{\d t} \E_k[F(t,X_t)]\right|
 \le  M  :=  \sup_{(t,\bx)\in[0,T]\times\mathbb{R}^d} \big|(\partial_t+\mathcal L)F(t,\bx)\big|.
\]
By the mean value theorem, for $t\in[t_k,t_{k+1}]$,
\[
\bigl|\E_k[F(t,X_t)]-F(t_k,X_k)\bigr|
 \le  \sup_{s\in[t_k,t_{k+1}]}\left|\tfrac{d}{ds} \E_k[F(s,X_s)]\right| (t-t_k)
 \le  M (t-t_k).
\]
Inserting this bound into \eqref{eq:Ek-abs} leads to
\[
|\mathcal{E}_k|  \le  \int_{t_k}^{t_{k+1}} M (t-t_k) dt
 =  \tfrac{1}{2} M (\Delta t)^2
 \le  C (\Delta t)^2,
\]
which establishes the claimed estimate.
\end{proof}

To ensure the stability of LLR estimator, it is crucial to establish nondegeneracy conditions for the weighted design matrix. The following two lemmas provide moment bounds and a spectral lower bound for the associated population covariance matrix.
\begin{lem}\label{pm}
Let $D_j := X_k^{j}-\bx\in\mathbb{R}^d$ and define radial weights
$w_j   :=   K \big(\|D_j\|/\varepsilon_k\big),$
where $K:[0,\infty)\to[0,\infty)$ is Lipschitz, compactly supported on $[0,1]$, and there exist constants $0<\rho\le 1$ and $K_{\min}>0$ such that $K(r)\ge K_{\min}$ for all $r\in[0,\rho]$; moreover $K(r)\le K_{\max}$ for all $r\ge 0$. 
Assume the sampling density $p$ on $\mathbb{B}_{\varepsilon_k}(\bx)$ is bounded and positive: $0<p_0 \le p(\xi)\le p_1<\infty,   \forall\xi\in\mathbb{B}_{\varepsilon_k}(\bx).$ 
Define the population moments
\begin{equation}\label{notation_xi}
\xi_0:=   \E[w_j],\quad \xi_1   :=   \E[w_j D_j],\quad \Sigma   :=   \E[w_j   D_jD_j^\top],
\end{equation}
where the expectation is taken with respect to the conditional law of $D_j$, whose density is proportional to $p(\bx+\xi) \mathbf{1}_{\{\|\xi\|\le \varepsilon_k\}}$ restricted to $\mathbb{B}_{\varepsilon_k}(0)$. 
If, in addition, the sampling is \emph{angularly symmetric} around $\bx$ (i.e., conditional on $\|D_j\|=r$, the direction $D_j/\|D_j\|$ is uniformly distributed on the unit sphere), then the following bounds hold:
\begin{equation}
\label{eq:xi0-bounds}
p_0K_{\min}   \mathrm{vol}(\mathbb{B}_{\rho\varepsilon_k}) \le \xi_0 \le p_1K_{\max}   \mathrm{vol}(\mathbb{B}_{\varepsilon_k}),\quad \xi_1   =   0, 
\end{equation}
and
\begin{equation}
\label{eq:sigma-lb}
\lambda_{\min}(\Sigma)  \ge   C_\Sigma   \varepsilon_k^{d+2}, \quad 
C_\Sigma:=\tfrac{\pi^{d/2}}{(d+2)\Gamma(d/2+1)}p_0K_{\min} \rho^{d+2}. 
\end{equation}
Here $\lambda_{\min}(\Sigma)$ denotes the \emph{smallest eigenvalue} of the symmetric positive semidefinite matrix $\Sigma$, and the volume of a $d$-dimensional ball of radius $r$ is $\mathrm{vol}(\mathbb{B}_r) = \omega_d r^d/d$, with  $\omega_d=\frac{2\pi^{d/2}}{\Gamma(d/2)}$.
\end{lem}

\begin{proof}
By the definition of $\xi_0$, setting $D := \xi - \bx$ gives
\begin{equation*}
\xi_0   =   \int_{\|D\|\le \varepsilon_k} K \Big(\frac{\|D\|}{\varepsilon_k}\Big) p(\bx+D) \d D.
\end{equation*}
For the lower bound in \eqref{eq:xi0-bounds}, we restrict to the region $\|D\|\le \rho\varepsilon_k$, where $K \ge K_{\min}$ and $p \ge p_0$, which yields
\begin{equation*}
\xi_0   \ge   p_0K_{\min}\int_{\|D\|\le \rho\varepsilon_k} \d D 
  =   p_0K_{\min}   \mathrm{vol}(\mathbb{B}_{\rho\varepsilon_k}).
\end{equation*}
For the upper bound, using $p \le p_1$ and $K \le K_{\max}$ on $\|D\|\le \varepsilon_k$ gives
\begin{equation*}
\xi_0   \le   p_1K_{\max}\int_{\|D\|\le \varepsilon_k}\d D 
  =   p_1K_{\max}   \mathrm{vol}(\mathbb{B}_{\varepsilon_k}).
\end{equation*}
Similarly, by the definition of $\xi_1$, we have
\begin{equation*}
\xi_1   =   \int_{\|D\|\le \varepsilon_k} K \Big(\frac{\|D\|}{\varepsilon_k}\Big) p(\bx+D)   D    \d D.
\end{equation*}
Under the \emph{angular symmetry} assumption (uniform directions conditional on radius), the angular integral of $D$ over any sphere $\{D:\|D\|=r\}$ is zero, while the weight $K(\|D\|/\varepsilon_k)$ depends only on $r$. Hence the integral vanishes and \eqref{eq:xi0-bounds} follows.

It remains to consider $\Sigma$, for which we have
\begin{equation*}
\Sigma   =   \int_{\|D\|\le \varepsilon_k} K \left(\frac{\|D\|}{\varepsilon_k}\right) p(x+D)   DD^\top    \d D 
  \succeq   p_0K_{\min} \int_{\|D\|\le \rho\varepsilon_k} DD^\top    \d D,
\end{equation*}
where $\succeq$ denotes the Loewner order on symmetric matrices. Exploiting isotropy of the integral, we obtain
\begin{equation*}
\begin{split}
\int_{\|D\|\le \rho\varepsilon_k} DD^\top   \d D &= \frac{1}{d} \Big(\int_{\|D\|\le \rho\varepsilon_k}  \|D\|^2  \d D\Big) I_d =  \frac{\omega_d}{d} \Big(\int_0^{\rho\varepsilon_k} r^2\cdot    r^{d-1}  \d r\Big) I_d\\
 &=  \frac{\omega_d}{d(d+2)} (\rho\varepsilon_k)^{d+2} I_d  =   \frac{\pi^{d/2}}{(d+2)   \Gamma(d/2+1)} (\rho\varepsilon_k)^{d+2} I_d,
\end{split}
\end{equation*}
where $\omega_d=\frac{2\pi^{d/2}}{\Gamma(d/2)}$ and  $I_{d}$ represents $d\times d$ identity matrix.
From the above two estimates we obtain 
$\Sigma \succeq C_{\Sigma} \varepsilon_k^{d+2} I_d$.
Therefore, by the Rayleigh--Ritz characterization, the smallest eigenvalue of the symmetric matrix $\Sigma$ satisfies
$
\lambda_{\min}(\Sigma) \ge C_\Sigma  \varepsilon_k^{d+2},
$
which proves the bound \eqref{eq:sigma-lb}.
\end{proof}

\begin{lem} 
Define $\eta_j:=w_j D_jD_j^\top\succeq 0$ and $\Gamma=\sum_{j=1}^M \eta_j$, so that $\mathbb E[\Gamma]=M\Sigma$. We introduce the event
\begin{equation}\label{defAk}
\mathcal A_k:=\Big\{\lambda_{\min}(\Gamma)\ge \frac12 C_{\mathcal{A}}   M   \varepsilon_k^{d+2}\Big\},  \quad  C_{\mathcal{A}} =\frac{\delta^2C_{\Sigma}}{2K_{max}}.
\end{equation}
There exist constants $C_{\mathcal{A}} >0$ such that
\begin{equation}\label{PAkc}
\mathbb P(\mathcal A_k^c)\ \le\ d   e^{-C_{\mathcal{A}} M\varepsilon_k^d}. 
\end{equation}
Moreover, on the event $\mathcal A_k$ one has
\begin{equation}\label{boundss}
\|\Gamma^{-1}\|\ \le\ \frac{2}{C_\Sigma M\varepsilon_k^{d+2}},\;\;\;
\|S\|\  = \| \sum_j w_j D_j \|  \le\ K_{\max}  M   \varepsilon_k^{d+1},\;\;\;
S_0=\sum_j w_j \asymp\ M   \varepsilon_k^d.
\end{equation}
where the constant $C_{\Sigma}$ is defined in \eqref{eq:sigma-lb}.
\end{lem}

\begin{proof}
We derive from \eqref{notation_xi} and \eqref{eq:sigma-lb} that
\begin{equation}\label{lambdagamma}
\lambda_{\min}(\mathbb E[\Gamma])=\lambda_{\min}(M\Sigma)\ \ge\ C_\Sigma M\varepsilon_k^{d+2}.
\end{equation}
Since $w_j\le K_{\max}$ and $\|D_j\|\le \varepsilon_k$, each summand $\eta_j=w_jD_jD_j^\top$ satisfies
$\lambda_{\max}(\eta_j)=\|\eta_j\| \le w_j\|D_j\|^2 \le K_{\max} \varepsilon_k^2.$ By the matrix Chernoff bound (cf.~\cite[Thm.~5.1]{Tropp2012}), for any $\delta\in(0,1)$,
\begin{equation*}
\mathbb{P} \Big\{\lambda_{\min}(\Gamma)\le (1-\delta) \lambda_{\min}(\E[\Gamma])\Big\}
\ \le\ d \exp \Big(-\frac{\delta^2}{2}\cdot\frac{\lambda_{\min}(\E[\Gamma])}{K_{\max}\varepsilon_k^2}\Big).
\end{equation*}
In particular,  inserting \eqref{lambdagamma} yields
\begin{equation*}
\mathbb{P} \Big\{\lambda_{\min}(\Gamma)\le (1-\delta) C_\Sigma M  \varepsilon_k^{d+2}\Big\}
\ \le\ d \exp \left(-\frac{\delta^2 C_\Sigma}{2K_{\max}} M\varepsilon_k^{d}\right)= d   \exp(- C_{\mathcal{A}}M\varepsilon_k^d),
\end{equation*}
With $\delta=\tfrac12$, we obtain 
$\mathbb P(\mathcal A_k^c)\le d e^{-C_{\mathcal{A}}M\varepsilon_k^d}$. On $\mathcal A_k$, the inverse bound 
$\lambda_{\min}(\Gamma)\ge \tfrac12 C_\Sigma M\varepsilon_k^{d+2}$ holds. Moreover,
\begin{equation*}
\|S\|=\Big\|\sum_j w_j D_j\Big\| \le \sum_j w_j\|D_j\| \ \le\ K_{\max} M \varepsilon_k^{d+1},
\end{equation*}
and the law of large numbers together with Lemma~\ref{pm} yields $S_0=\sum_j w_j\asymp M\varepsilon_k^d$.
\end{proof}

We proceed to a rigorous analysis of the error in the Taylor expansion \eqref{taylorexp}, where the first-order truncation plays a crucial role by directly linking the known solution $u(t_{k+1},\cdot)$ (approximated by $\{\Y^j_{k+1}\}_j$) with the gradient $\nabla u$, thereby enabling the particle-based LLR construction and yielding the gradient approximation $\alpha_{\bx}$.
\begin{lem} {\rm \bf(Taylor truncation error \eqref{taylorexp})} \label{TRB}
Let $\X_k^m=\bx$ in \eqref{taylorexp}, so that the next point can be written as $\bx+D_j$. 
Let $u(t_{k+1},\cdot)\in C^2(\mathbb B_{\varepsilon_k}(\bx))$ with
$\|\nabla^2 u(t_{k+1},\cdot)\|_\infty$  $\le C_{\nabla^2}$. For each $j$, introduce the Taylor remainder
\begin{equation}\label{defrj}
r_j=u(t_{k+1},\bx+D_j)-u(t_{k+1},\bx) -\partial_t u(t_k, \bx)   \Delta t- \nabla u(t_{k+1},\bx)^\top D_j,
\end{equation}
Then the following estimate holds
\begin{equation}\label{boundrj}
|r_j|\ \le\ \tfrac12   \|\nabla^2 u\|_\infty   \|D_j\|^2\ \le\ \tfrac12 C_{\nabla^2}\varepsilon_k^2.
\end{equation}
Moreover, we have
\begin{equation}\label{wrwrdbound}
\sum_{j=1}^M w_j|r_j|\ \le\ C   M   \varepsilon_k^{d+2},\qquad
\Big\|\sum_{j=1}^M w_j r_j D_j\Big\|\ \le\ C   M   \varepsilon_k^{d+3}.
\end{equation}
\end{lem}
\begin{proof}
We apply the second-order Taylor expansion of $u(t_{k+1},\cdot)$ at $\bx$ in the direction $D_j$, which yields
\begin{equation*}
u(t_{k+1},\bx+D_j) = u(t_{k+1},\bx) +\partial_t u(t_k, \bx)   \Delta t+ \nabla u(t_{k+1},\bx)^\top D_j + r_j,
\end{equation*}
where the remainder takes the integral form
\begin{equation*}
r_j=\int_0^1 (1-s) D_j^\top \big(\nabla^2 u\big)(t_{k+1},\bx+sD_j) D_j {\rm d}s.
\end{equation*}
Since $\|D_j\|\le \varepsilon_k$ and $\|\nabla^2 u\|_\infty\le C_{\nabla^2}$, it follows that 
$|r_j|\le \tfrac12 C_{\nabla^2}\varepsilon_k^2$. Consequently, from \eqref{boundss} we deduce
\begin{equation*}
\sum_j w_j|r_j|\ \le\ \tfrac12 C_{\nabla^2} \sum_j w_j \varepsilon_k^2\ \le\ C \cdot M\varepsilon_k^d \cdot \varepsilon_k^2 = C M \varepsilon_k^{d+2}.
\end{equation*}
Similarly, invoking \eqref{eq:sigma-lb}, we obtain
\begin{equation*}\begin{split}
\Big\|\sum_j w_j r_j D_j\Big\|
& \le\ \sum_j w_j |r_j| \|D_j\|
\\&\le\ \big(\tfrac12 C_{\nabla^2}\varepsilon_k^2\big)\cdot \sum_j w_j \|D_j\|
 \le\ C \varepsilon_k^2\cdot M\varepsilon_k^{d+1}
= C M \varepsilon_k^{d+3}.
\end{split}\end{equation*}
This completes the proof.
\end{proof}
A direct analysis of the error between the numerical solution $\alpha_{\bx}$ and the gradient $\nabla u$ is rather difficult. To address this, we first introduce an auxiliary least-squares solution $\alpha^\star_{\bx}$ by incorporating the Taylor remainder term $r_j$, and then establish bounds for the associated Schur complement matrix of the auxiliary problem, thereby preparing the ground for the subsequent analysis of $\alpha_{\bx}-\alpha^\star_{\bx}$.
\begin{lem}{\rm\bf(Bounds for Schur complement matrix)}
\label{SCB}
Let $D_j := X_k^j - \bx \in \mathbb{R}^d$ denote local displacements around an anchor $\bx$, and
let kernel weights be $w_j = K(\|D_j\|/\varepsilon_k)$ with a bounded kernel $K$ supported on $[0,1]$.
Define the weighted moments
\[
S_0 := \sum_{j=1}^M w_j,\qquad
S := \sum_{j=1}^M w_j D_j,\qquad
\Gamma := \sum_{j=1}^M w_j D_j D_j^\top.
\]
Define the auxiliary \emph{noiseless} responses  
\begin{equation}\label{ykj1}
Y_{k+1}^j := u(t_{k+1}, \bx + D_j)  =  \alpha^\star  +  (\alpha^\star_{\bx})^\top D_j  +  r_j,
\end{equation} 
together with the corresponding weighted least-squares minimizers $(\alpha^\star,\alpha_{\bx}^\star)$, where the Taylor remainder $r_j$ is defined in \eqref{defrj}.  
Let $(\alpha,\alpha_{\bx})$ denote the weighted least-squares minimizers associated with $\{\Y_{k+1}^j\}$ as in \eqref{WLS_objective}. Then their differences, defined as $\delta^{j}_{k+1} := \Y_{k+1}^j - Y_{k+1}^j$, can be expressed as
\begin{equation}
\label{eq:schur-noisy}
\bigl(\Gamma - S S_0^{-1} S^\top \bigr) (\alpha_{\bx} - \alpha^\star_{\bx})
 =  \sum_{j=1}^M w_j (\delta^{j}_{k+1}+r_j) D_j  -  S S_0^{-1} \sum_{j=1}^M w_j (\delta^{j}_{k+1}+r_j),
\end{equation}
and  
\begin{equation}
\label{eq:alpha0-noisy}
\alpha - \alpha^\star  =  S_0^{-1}\bigg(\sum_{j=1}^M w_j (\delta^{j}_{k+1}+r_j) - S^\top(\alpha_{\bx}- \alpha^\star_{\bx})\bigg).
\end{equation}
Moreover, there exists a constant $c_{\rm sch}>0$ such that, on the event $\mathcal{A}_k$,
\begin{equation}
\label{eq:schur-inv-bound}
\lambda_{\min} \bigl(\Gamma - S S_0^{-1} S^\top\bigr) \ge  c_{\rm sch}  M  \varepsilon_k^{d+2},
\qquad
\bigl\|(\Gamma - S S_0^{-1} S^\top)^{-1}\bigr\|  \le  \frac{1}{c_{\rm sch} M \varepsilon_k^{d+2}},
\end{equation}
and the complement satisfies $\mathbb{P}(\mathcal{A}_k^c)\le d \exp(-C_{\mathcal{A}} M \varepsilon_k^d)$, where $\mathcal{A}_k$ and $C_{\mathcal{A}}$ are defined in \eqref{defAk}.
\end{lem}

\begin{proof}
For clarity, we first rewrite the linear system \eqref{normal-eq} obtained from the weighted least-squares minimizers \eqref{WLS_objective}, together with its counterpart corresponding to \eqref{ykj1}, into a Schur complement matrix representation
\[
\begin{pmatrix} S_0 \;S^\top \\ \hspace{-3pt}S \;\;\;\Gamma \end{pmatrix}
\begin{pmatrix} \alpha \\ \alpha_{\bx} \end{pmatrix}
=\begin{pmatrix} \sum_j w_j \Y_{k+1}^j \\ \sum_j w_j \Y_{k+1}^j D_j \end{pmatrix},
\;\;
\begin{pmatrix} S_0 \;S^\top \\ \hspace{-3pt}S \;\;\; \Gamma \end{pmatrix}
\begin{pmatrix} \alpha^\star \\ \alpha_{\bx}^\star \end{pmatrix}
=
\begin{pmatrix} \sum_j w_j (Y_{k+1}^j-r_j) \\ \sum_j w_j (Y_{k+1}^j-r_j) D_j \end{pmatrix}.
\]
Subtracting the two systems yields
\[
\begin{pmatrix} S_0 & S^\top \\ S & \Gamma \end{pmatrix}
\begin{pmatrix} \alpha - \alpha^\star \\ \alpha_{\bx} - \alpha^\star_{\bx} \end{pmatrix}
=
\begin{pmatrix} \sum_j w_j (\delta^{j}_{k+1}+r_j) \\ \sum_j w_j (\delta^{j}_{k+1}+r_j) D_j \end{pmatrix}.
\]
By applying the standard Schur complement procedure, we readily obtain \eqref{eq:schur-noisy} and \eqref{eq:alpha0-noisy}.

For the spectral bound, observe that for any $v\in\mathbb R^d$, 
\begin{equation*}
v^\top \big(SS_0^{-1}S^\top\big) v\ \le\ \frac{\|S\|^2}{S_0}   \|v\|^2,
\end{equation*}
which implies
\begin{equation*}
\lambda_{\min}\big(\Gamma - S S_0^{-1} S^\top\big)
\ \ge\ \lambda_{\min}(\Gamma)\ -\ \frac{\|S\|^2}{S_0}.
\end{equation*}
On the event $\mathcal A_k$,  we find from \eqref{boundss} that
\begin{equation*}
\frac{\|S\|^2}{S_0}\ \le\ \frac{K^2_{\max} M^2\varepsilon_k^{2d+2}}{c   M\varepsilon_k^{d}}
\ =C M   \varepsilon_k^{d+2}.
\end{equation*}
By choosing $M$ sufficiently large (or absorbing constants into $C$), we may fix $c_{\mathrm{sch}} := \tfrac14 C_{\mathcal{A}} > 0$ such that
\begin{equation}\label{spectralrad}
\lambda_{\min}\big(\Gamma - S S_0^{-1} S^\top\big)
\ \ge\ \big(\frac12 C_{\mathcal{A}} - C\big)M\varepsilon_k^{d+2}\ \ge\ c_{\mathrm{sch}} M\varepsilon_k^{d+2}.
\end{equation}
Finally, the corresponding inverse bound follows directly as the reciprocal of this minimal eigenvalue.
\end{proof}

\begin{rem}
{ We note that reusing common randomness across particles may induce weak correlations in $\{\delta_{k+1}^j\}$. Such correlations only modify constants in the variance via an effective-sample-size factor and do not alter the rate in \eqref{eq:grad-bv}. For clarity, we adopt the standard i.i.d. assumption in Lemma~\ref{GEEB}; this assumption holds if we draw fresh auxiliary simulations for each particle at every time level.} 
\end{rem}

The following lemma provides error estimates for the weighted least-squares minimizer $\alpha_{\bx}$ (see \eqref{WLS_objective}) in comparison with the exact gradient $\nabla u$, which play a central role in the final error analysis of $\Y$.
\begin{lem}{\rm\bf(Error bound for the gradient estimator $\nabla u$)}
\label{GEEB} 
Let $\alpha_{\boldsymbol{x}}$ denote the finite-sample minimizer of \eqref{WLS_objective} associated with
$\Y_{k+1}^j = Y_{k+1}^j + \delta^{j}_{k+1}$, where $\delta^{j}_{k+1}$ represents the error in $\Y^j_{k+1}$.
Assume that, conditional on $\mathcal{F}_{t_k}$, the error terms $\{\delta_{k+1}^j\}_j$ are independent and identically distributed. 
Then it holds that
\begin{equation} \begin{split}\label{eq:grad-bv}
\E_{k} \left[\bigl\|\alpha_{\boldsymbol{x}}- \nabla u(t_k,\boldsymbol{x})\bigr\|^2\right] 
&\le C \varepsilon_k^2  +  C \varepsilon_{k}^{-2} \E_{k}\big[|\delta_{k+1}|^{2}\big]
 +  C e^{-C_{\mathcal{A}_k} M\varepsilon_k^d},
\end{split}
\end{equation}
where $C$ is a positive constant independent of $\varepsilon_k$ and $M$, and $C_{\mathcal{A}_k}$ is defined in \eqref{defAk}.
\end{lem}
\begin{proof}
We introduce the \textbf{ideal} least-squares solution  $\alpha_{\bx}^{\star}$ (cf. \eqref{ykj1}) corresponding to the noise-free case and decompose the error into bias and variance components:
\begin{equation*}
\E_{k}\Big[\bigl\|\alpha_{\boldsymbol{x}}- \nabla u(t_k,\boldsymbol{x})\bigr\|^2\Big]\le 2 \E_{k}\Big[\bigl\|\alpha_{\boldsymbol{x}}^\star - \nabla u(t_k,\boldsymbol{x})\bigr\|^2\Big]
 +  2 \E_{k} \left[\bigl\|\alpha_{\boldsymbol{x}} - \alpha^\star_{\boldsymbol{x}}\bigr\|^2\right].
\end{equation*}
In fact, in the limit $M \to \infty$, it follows from \eqref{ykj1} that $\nabla u(t_k,\bx)$ coincides with the optimal solution of the weighted regression. Hence outside the event $\mathcal{A}_k$ (sufficient sampling within the $\varepsilon_k$-ball), the contribution is negligible. More intuitively, as long as $M\varepsilon_k^d$ is large enough, the probability of the event $\mathcal{A}_k^c$ with a lack of samples in the neighborhood will rapidly decay at the rate of $de^{-C_{\mathcal{A}_k} M\varepsilon_k^d}$. Therefore, when estimating the error, the contribution of this tail event can be safely ignored, and only an additional $d e^{-C_{\mathcal{A}_k} M \varepsilon_k^d}$ term needs to be added to cover it. In the following we restrict to $\mathcal{A}_k$, ignoring the exponentially small complement.

We now turn to the estimation of the second term $|\alpha_{\bx} - \alpha^\star_{\bx}|$.
By \eqref{eq:schur-noisy}, we have
$$\alpha_{\bx}- \alpha_{\bx}^\star 
= (\Gamma - S S_0^{-1} S^\top)^{-1}
\Big(\sum_j w_j (\delta^{j}_{k+1}+r_j) D_j  -  S S_0^{-1}\sum_j w_j (\delta^{j}_{k+1}+r_j)\Big).$$
Under the condition that $\{\delta^{j}_{k+1}\}_j$ are independent and identically distributed, then on $\mathcal{A}_k$, we have
\begin{equation*}
\E_{k}\left[\delta^{j}_{k+1}\right] = \E_k\big[\delta_{k+1}\big],  \quad \E_{k}\left[|\delta^{j}_{k+1} |^{2}\right] =\E_{k}\left[|\delta_{k+1}|^{2}\right] ,  
\end{equation*}
which implies 
$$
\E_{k}\bigg[ \Bigl\|\sum_j w_j (\delta^{j}_{k+1}+r_j) D_j\Bigr\|^2 \bigg]
\leq  \E_{k}\left[|\delta_{k+1}|^{2}\right] \Big(\sum_j w_j D_j \Big)^{2}+ \E_{k}\bigg[\Big\|\sum_j w_j r_jD_j \Big\|^{2}\bigg] .
$$
Since $w_j \le K_{\max}$ and $\|D_j\|\le \varepsilon_k$, we obtain that
$$\sum_j w_j D_j  \le K_{\max} M \varepsilon_k^{d+1}.$$
From the above inequality and \eqref{wrwrdbound}, it follows that
\[
\E_{k}\bigg[\Bigl\|\sum_j w_j \delta^{j}_{k+1} D_j\Bigr\|^2 \bigg]\le  (K_{\max} M )^{2} \varepsilon_k^{2d+2} \E_{k}\left[|\delta_{k+1}|^{2}\right] +M^2   \varepsilon_k^{2d+6} .
\]
A similar bound holds for the $S S_0^{-1}$ term,
\begin{equation*}\begin{split}
\E_{k}\bigg[\Bigl\| S S_0^{-1} \sum_j w_j (\delta^{j}_{k+1}+r_j)\Bigr\|^2\bigg]
&\leq \E_{k}[(\delta_{k+1})^{2}] \Big(\varepsilon_{k}\sum_j w_j\Big)^2+\E_{k}\bigg[\Big\|\varepsilon_{k}\sum_j w_j r_j\Big\|^{2}\bigg] 
\\&\le  (K_{\max} M )^{2}  \varepsilon_k^{2d+2} \E_{k}\left[|\delta_{k+1}|^{2}\right] +M^2\varepsilon_k^{2d+6}.
\end{split}\end{equation*}
Using the spectral bound \eqref{spectralrad} for $(\Gamma - S S_0^{-1} S^\top)^{-1}$, we deduce on $\mathcal{A}_k$,
$$\E_{k} \left[\|\alpha_{\bx}- \alpha^\star _{\bx}\|^2\right]
\le \frac{K_{\max}^{2}}{c^{2}_{sch}}\varepsilon_{k}^{-2} \E_{k}\left[|\delta_{k+1}|^{2}\right]+\frac{K_{\max}^{2}}{c^{2}_{sch}} \varepsilon_k^2  .$$
On the complement event $\mathcal{A}_k^c$, we employ a crude envelope bound weighted by the exponentially small probability
$\mathbb{P}(\mathcal{A}_k^c)\le d e^{-C_{\mathcal{A}_k} M \varepsilon_k^d}$.
Taking expectations and combining the results on the events $\mathcal{A}_k$ and $\mathcal{A}_k^c$ then yields \eqref{eq:grad-bv}.
\end{proof}

 Denote $\operatorname{Var}_k(\cdot):=\operatorname{Var}(\cdot\,|\,\mathcal F_{t_k})$ as the conditional variance with respect to the filtration at time $t_k$. Then, we obtain the following  conditional variance bound.
\begin{lem}{\rm\bf(Conditional variance bound for $Y_{k+1}$)}\label{lem:one-step-cond-var} 
Under {\rm Assumption~\ref{ass:HLG}},  then we have
\begin{equation}\label{varky}
\operatorname{Var}_k \left(Y_{k+1}\right) \le\ C\,\Delta t,
\end{equation}
where the positive constant $C$ independent of $\Delta t$. 
\end{lem}
\begin{proof}
For notational simplicity, set $\tilde{f}(s):=f(s,X_s,Y_s,Z_s)$ in this proof. Recall the BSDE on $[t_k,t_{k+1}]$:
\[
Y_{k+1} = Y_{k} - \int_{t_k}^{t_{k+1}} \tilde{f}(s)\,\d s+ \int_{t_k}^{t_{k+1}} Z_s\,\d W_s.
\]
Taking conditional expectation with respect to $\mathcal F_{t_k}$ on both sides and subtracting, and
using $\E_k\!\big[\int_{t_k}^{t_{k+1}} Z_s\,\mathrm{d}W_s\big]=0$ together with conditional Fubini, we obtain that
\begin{equation*}\begin{split}
{\rm Var}_k(Y_{k+1})&= \E_k\Big[\left(Y_{k+1}-\E_k[Y_{k+1}]\right)^2 \Big] \\&
= \mathbb{E}_k\Big[\Big(\int_{t_k}^{t_{k+1}} Z_s\,\mathrm{d}W_s
- \int_{t_k}^{t_{k+1}}\big(\tilde{f}(s)-\mathbb{E}_k \tilde{f}(s)\big)\,\mathrm{d}s\Big)^2\Big] \\
&\le 2\,\mathbb{E}_k\Big[\Big(\int_{t_k}^{t_{k+1}} Z_s\,\mathrm{d}W_s\Big)^2\Big]
   + 2\,\mathbb{E}_k\Big[\Big(\int_{t_k}^{t_{k+1}}\big(\tilde{f}(s)-\mathbb{E}_k \tilde{f}(s)\big)\,\mathrm{d}s\Big)^2\Big] \\
&= 2\,\mathbb{E}_k\int_{t_k}^{t_{k+1}} \|Z_s\|^2\,\mathrm{d}s  + 2\,\mathbb{E}_k\Big[\Big(\int_{t_k}^{t_{k+1}}\big(\tilde{f}(s)-\mathbb{E}_k \tilde{f}(s)\big)\,\mathrm{d}s\Big)^2\Big] \\
&\le 2\,\mathbb{E}_k\int_{t_k}^{t_{k+1}} \|Z_s\|^2\,\mathrm{d}s
   + 2\,\Delta t \int_{t_k}^{t_{k+1}} \mathbb{E}_k\left[ \big|\tilde{f}(s)-\mathbb{E}_k \tilde{f}(s)\big|^2 \right] \,\mathrm{d}s \\
&\le 2\,\mathbb{E}_k\int_{t_k}^{t_{k+1}} \|Z_s\|^2\,\mathrm{d}s
   + 2\,\Delta t \int_{t_k}^{t_{k+1}} \mathbb{E}_k\left[ \big|\tilde{f}(s)\big|^2 \right] \,\mathrm{d}s.
\end{split}\end{equation*} 
From the standard a priori estimate $\sup_{s\le T}\mathbb{E}\|Z_s\|^2\le C$ it follows that the first term is $\le C\Delta t$, which yields the desired result.
\end{proof}

 In fact, the family of numerical solution $\{\Y_k^{j}\}_{j=1}^M$ is conditionally exchangeable given $\mathcal F_{t_k}$ rather than independent, since each $\Y_k^{j}$ is formed via partial averaging of given data $\{\Y_{k+1}^{j}\}_{j=1}^M$. The next lemma quantifies the resulting correlation.
 \begin{lem}\label{lem:mc-coupled}
Assume that the particles $\{\Y_{k+1}^{j}\}_{j=1}^M$ are conditionally exchangeable given $\mathcal F_{t_k}$. 
Let
\begin{equation}\label{rhoMk}
\bar\rho_k := \frac{2}{M(M-1)}\sum_{1\le j<\ell\le M}
\operatorname{Corr}_k \Big(\Y_{k+1}^{j},\Y_{k+1}^{\ell}\Big),
\quad
M_{\rm eff}(k):=\frac{M}{1+(M-1)\bar\rho_k}.
\end{equation}
Define $$\xi_{k+1}:=\frac 1 M \sum_{j=1}^M\Y_{k+1}^{j}-\mathbb E_k\big[\Y_{k+1}\big],$$ 
then 
\begin{equation}\label{eq:xi2-coupled-exp}
\mathbb E_{k}[|\xi_{k+1}|^2]\ \le\ \frac{C\,\Delta t}{M_{\rm eff}(k)}\ +\ \frac{C}{M_{\rm eff}(k)}\,\mathbb E
_{k}\big[|\Y_{k+1}-Y_{k+1}|^2\big].
\end{equation}
\end{lem}
\begin{proof}
Set $\operatorname{Corr}_k(\cdot):=\operatorname{Corr}(\cdot\,|\,\mathcal{F}_{t_k})$. Define the centered variables $U_j:=\Y^{j}_{k+1}-\mathbb{E}_k[\Y^{j}_{k+1}]$ with $\mathbb{E}_k[U_j]=0$, and define the conditional pairwise correlations $\rho_{j\ell,k}:=\operatorname{Corr}_k\!\big(\Y^{j}_{k+1},\,\Y^{\ell}_{k+1}\big)$ for $j\neq \ell$.
One can verify easily that $$\bar\rho_k\ :=\ \frac{2}{M(M-1)}\sum_{1\le j<\ell\le M}\rho_{j\ell,k}\in\big[-\tfrac{1}{M-1},\,1\big].$$
By the definition of $U_j$ we deduce that
\begin{equation}\label{proof1}
\begin{split}\operatorname{Var}_k \bigg(\frac 1 M \sum_{j=1}^{M}\Y^{j}_{k+1}\bigg)
&=\E_k \Big[\Big(\frac1M\sum_{j=1}^M U_j\Big)^2\Big]
\\&=\frac{1}{M^2}\sum_{j=1}^M \E_k[U_j^2]+\frac{2}{M^2}\sum_{1\le j<\ell\le M}\E_k[U_jU_\ell].
\end{split}\end{equation}
Conditional exchangeability implies $\E_k[U_j^2]=\operatorname{Var}_k \big(\Y_{k+1}\big)$ for all $j$, and
$$
\E_k[U_jU_\ell]=\operatorname{Cov}_k\big(\Y^{j}_{k+1},\Y^{\ell}_{k+1}\big)
=\rho_{j\ell,k}\operatorname{Var}_k \big(\Y_{k+1}\big),\quad j\ne \ell.
$$
Hence
\begin{equation*}
\begin{split}
\operatorname{Var}_k \bigg(\frac 1 M \sum_{j=1}^{M}\Y^{j}_{k+1}\bigg)
&=\frac{1}{M^2}\Big(M\operatorname{Var}_k\big(\Y_{k+1}\big) +2\operatorname{Var}_k\big(\Y_{k+1}\big)\sum_{1\le j<\ell\le M}\rho_{j\ell,k}\Big)\\
&=\frac{\operatorname{Var}_k\big(\Y_{k+1}\big)}{M^2}\Big(M+M(M-1)\bar\rho_k\Big),
\end{split}
\end{equation*}
because $\sum_{j<\ell}\rho_{j\ell,k}=\tfrac{M(M-1)}{2}\bar\rho_k$ by the definition of $\bar\rho_k$ and covariance decomposition for a correlated mean,
\begin{equation}\label{proof2}
\operatorname{Var}_k \bigg(\frac1M\sum_{j=1}^M \Y_{k+1}^{j}\bigg)
=\frac{1+(M-1)\bar\rho_k}{M}\operatorname{Var}_k \big(\Y_{k+1}\big)
=\frac{1}{M_{\rm eff}(k)}\operatorname{Var}_k \big(\Y_{k+1}\big).
\end{equation}
Write $\Y_{k+1}=Y_{k+1} + (\Y_{k+1}-Y_{k+1})$ and apply
$(a+b)^2\le 2a^2+2b^2$ conditionally:
\begin{equation}\label{proof3}\begin{split}
\operatorname{Var}_k \big(\Y_{k+1}\big)
\le 2\,\operatorname{Var}_k \big(Y_{k+1}\big)+ 2\,\mathbb E_k \left[|\Y_{k+1}-Y_{k+1}|^2\right]. 
\end{split}\end{equation}
In view of \eqref{proof1}–\eqref{proof3}, we obtain that
\begin{equation*}\begin{split}
\mathbb E_k \left[|\xi_{k+1}|^2 \right]&=\E_k \Big[\Big(\frac1M\sum_{j=1}^M U_j\Big)^2\Big]
= \frac{1}{M_{\rm eff}(k)}\ \operatorname{Var}_k \big(\Y_{k+1} \big)
\\& \le\ \frac{2}{M_{\rm eff}(k)}\Big(
\operatorname{Var}_k(Y_{k+1}) + \mathbb E_k[|\Y_{k+1}-Y_{k+1}|^2]\Big).
\end{split}\end{equation*}
Finally, the above equation and \eqref{varky} directly imply the desired result.
\end{proof}

 We are now ready to present the error estimate of the final numerical solution $\Y_0$.

\begin{thm}{\rm\bf(Global Error)}\label{ITE}
Define the error $\delta_k := \Y_k - Y_k$ for $0\le k\le N$. 
Suppose Assumption~{\rm\ref{ass:HLG}} and the hypotheses of Lemma~{\rm\ref{lem:DEES}} are satisfied. 
Then, for all sufficiently large $M$, we have
\begin{equation}\label{finalboundMC}
\E\big[|\delta_0|^2\big] \le   C\,\Delta t   +   C\,\Delta t\,e^{-c_1 M},
\end{equation}
where $C>0$ depends only on $T$, the Lipschitz constant $L$, but is independent of $\Delta t$, while the constant $c_1$ depends on $\varepsilon_k$.
\end{thm}

\begin{proof}
Along the forward particles  the exact solution \eqref{Ytk} satisfies 
\begin{equation*}
Y_{k}= \E_k \Big[Y_{k+1}+\int_{t_k}^{t_{k+1}} f\big(s,X_s,Y_{s},Z_{s}\big)\,\d s\Big],
\end{equation*}
and its implemented time-discrete approximation \eqref{discreteYkm} reads
\begin{equation*}
\Y_k   =   \frac 1 M \sum_{j=1}^{M} \Y^{j}_{k+1}
 +  \Delta t\, f\big(t_k,X_k,\Y_k,\Z_{k}\big).
\end{equation*}
Subtracting the above two relations and incorporating the remainder estimate \eqref{c}, we obtain
\begin{equation}\label{deltak1}
\begin{split}
\delta_k  =&   \frac 1 M \sum_{j=1}^{M} \Y^{j}_{k+1}- \E_k \left[Y_{k+1}\right] + f(t_k,X_k,\Y_k,\Z_{k})\Delta t  -\E_k \Big[\int_{t_k}^{t_{k+1}} f\big(s,X_s,Y_{s},Z_{s}\big)\,\d s\Big]\\
=&\Big(\frac 1 M \sum_{j=1}^{M} \Y^{j}_{k+1} - \E_k \big[\Y_{k+1}\big]\Big) + \Big(\E_k \big[\Y_{k+1}\big]- \E_k \big[Y_{k+1}\big]\Big)\\
& +\left( f(t_k,X_k,\Y_k,\Z_{k}) -f(t_k,X_k,Y_k,Z_{k})    \right)\Delta t\\
&  + \Big( f(t_k,X_k,Y_k,Z_{k})\Delta t  -\E_k \Big[\int_{t_k}^{t_{k+1}} f\big(s,X_s,Y_{s},Z_{s}\big)\,\d s\Big] \Big)
\\&= \xi_{k+1}+\E_k \big[\delta_{k+1}\big] +\left( f(t_k,X_k,\Y_k,\Z_{k}) -f(t_k,X_k,Y_k,Z_{k})    \right)\Delta t+\mathcal{E}_{k}.
\end{split}
\end{equation}
Then, the last term in the above equation can be bounded by the Lipschitz condition
\begin{equation*}
\begin{split}
&\big|f(t_k,X_k,\Y_k,\Z_{k}) - f(t_k,X_k,Y_k,Z_k)\big| \\
&= \big|f(t_k,X_k,\Y_k,\sigma^{\top}\alpha_{\bx,k}) - f(t_k,X_k,Y_k,\sigma^{\top}\nabla u_k)\big|
\le L\big(|\delta_k|+\|\sigma\|\cdot\|\alpha_{\bx,k}-\nabla u_k\|\big),
\end{split}
\end{equation*}
which together with \eqref{deltak1} leads to
\begin{equation*}
\begin{split}
(1-L\Delta t) |\delta_k|  \le &  \E_k \left[| \delta_{k+1} |\right]  + L\Delta t \big(\|\sigma\|\cdot\|\alpha_{\bx,k}-\nabla u_k\|\big)+ | \xi_{k+1} |+ |\mathcal{E}_{k}|.
\end{split}
\end{equation*}
Hence, by conditional Jensen's inequality, $|\E_k[\delta_{k+1}]|^2 \le \E_k[|\delta_{k+1}|^2]$, together with \eqref{c}, it follows that, for any $\eta>0$, 
\begin{equation}\label{eq:ABC-weighted-MC}
\begin{split}
(1-L\Delta t)^{2}\E_k \big[|\delta_k|^2\big] \le & (1+\eta)\E_k[|\delta_{k+1}|^2]+ C_\eta   (L\Delta t\|\sigma\|)^2\,\E_k \big[\|\alpha_{\bx,k}-\nabla u_k\|^2\big] \\
&+ \mathbb E_{k}\big[|\xi_{k+1}|^2\big] + C(\Delta t)^{4},
\end{split}
\end{equation}
where $C_\eta$ is a positive constant depends on $\eta$.
A combination of \eqref{eq:ABC-weighted-MC}, \eqref{eq:grad-bv}, and \eqref{eq:xi2-coupled-exp} leads to 
\begin{equation}\label{eq:recursion-MC}
\E\big[|\delta_k|^2\big]
 \le  \gamma_k\,\E\big[|\delta_{k+1}|^2\big]  +  \beta_k,
\end{equation}
with
\begin{equation}\label{eq:gamma-beta-MC}
\begin{split}
&\gamma_k  =  \frac{(1+\eta) + C_\eta\,(L\Delta t\|\sigma\|)^2\,\varepsilon_k^{-2} +\frac{C_{\eta}}{M_{\rm eff}(k)}}{(1-L\Delta t)^2},\\
&\beta_k  =  \frac{C_\eta}{(1-L\Delta t)^2} \Big((L\Delta t\|\sigma\|)^2\big(\varepsilon_k^2 + e^{-c_1 M\varepsilon_k^d}\big)+ (\Delta t)^{4} + \frac{\Delta t}{M_{\rm eff}(k)}\Big).
\end{split}
\end{equation}
By a Taylor expansion, for sufficiently small $\Delta t$ we obtain $ (1-L\Delta t)^{-2} \le 1 + C\,\Delta t. $
When the radius is a constant $\varepsilon_k\in(0,1]$, so that $\Delta t\,\varepsilon_k^{-2}=O(\Delta t)$, we set $\eta:=L\Delta t$ and applying \eqref{eq:gamma-beta-MC} gives
$$\gamma_k \;\le\; \bigl(1 + C\Delta t\bigr)\,\Bigl(1 + C\,\Delta t\,\varepsilon_k^{-2}\Bigr)\le(1+C\Delta t)(1+C\Delta t) \;\le\; 1 + C\Delta t.$$
Otherwise, when the radius is small with $\varepsilon_k\asymp \sqrt{\Delta t}$, we have $\varepsilon_k^{-2}\asymp \Delta t^{-1}$, hence $(L\Delta t)^2\varepsilon_k^{-2}=L^2\Delta t=\mathcal{O}(\Delta t)$, and to keep $\gamma_k$ in the form $1+C\Delta t$ we choose a constant $\eta\in(0,1]$, whence
\[
\gamma_k \le\; 1 + C\Delta t.
\]
%
On the other hand, it is known from $\gamma_k \le 1 + C\Delta t$ that
$$\log\Big(\prod_{j=k}^{N-1}\gamma_j \Big) = \sum_{j=k}^{N-1}\log(\gamma_{j}) \le \sum_{j=k}^{N-1}(\gamma_{j}-1)\le \sum_{j=k}^{N-1}C\Delta t = C(T-t_{k}),$$
which implies
$$\Big(\prod_{j=k}^{N-1}\gamma_j \Big) \le \exp( \sum_{j=k}^{N-1}C\Delta t )\le e^{C(T-t_{k})}.$$
Using the fact that $\delta_N=0$, the discrete Gronwall Lemma \ref{le3.2} and \eqref{eq:recursion-MC} yield
$$
\E\big[|\delta_0|^2\big]
 \le  e^{CT}\sum_{k=0}^{N-1}\beta_k
 \le  C\sum_{k=0}^{N-1}\Big((L\Delta t)^2\,\varepsilon_k^2 + (L\Delta t\|\sigma\|)^2 e^{-c_1 M\varepsilon_k^d} + (\Delta t)^4 + \frac{\Delta t}{M_{\rm eff}(k)}\Big).
$$
Since the last term $\Delta t\sum_{k=0}^{N-1}\frac{1}{M_{\rm eff}(k)}$ is of order $O(\Delta t)$ for sufficiently large $M$, the desired bound follows.
 This ends the proof.
\end{proof}

%
%
%

\section{Numerical experiments}\label{numexp}
In this section, we present several representative numerical experiments in very high dimensions to verify the accuracy, efficiency, and stability of the proposed stochastic algorithm. 
We employ contrived analytic solutions to demonstrate the temporal convergence rates of the proposed methods. 
It is worth noting that the test cases cover a range of challenging scenarios, including strong nonlinearity, gradient dependence, and problem dimensions up to $10000$. 
All experiments were performed on a personal laptop 
\texttt{MacBook Pro (model Z15H000THCH/A), Apple M1 Pro chip (10 cores: 8 performance + 2 efficiency), 32 GB unified } \\ \texttt{memory, macOS system firmware version 10151.140.19}.

\subsection{Allen-Cahn equation}
We first consider the Allen--Cahn equation in high dimensions, a classical reaction--diffusion model in physics that serves as a prototype for phase separation and order--disorder transitions.
\begin{equation}
\label{ACeq}
\partial_{t} u(t,\bx) + \Delta u(t,\bx)  + f(u)=  0,     \;\;\; (t,\bx)\in[0,T)\times\mathbb{R}^d.
\end{equation}
 In our experiments, we study two cases with different nonlinear terms.
\begin{description}
\item[Case 1.] Double-well potential $f(u)=u-u^3$ and terminal condition $u(T,\bx)=1/(2+0.4\|\bx\|^2)$, with $\bx\in\mathbb{R}^d$.
\item[Case 2.] Logarithmic potential $f(u)=\frac{\theta}{2}\ln(\frac{1+u}{1-u})-\theta_cu$ with $\theta<\theta_c$ are two positive constants. 
To facilitate numerical validation, we construct a manufactured solution  
$$
u(t,\bx) =\cos\Bigl(\prod_{j=1}^d x_j\Bigr) {\rm e}^{\cos t-\|\bx\|^2}, \quad \bx\in\mathbb{R}^d,
$$
by adding an external source term on the right-hand side.
\end{description}

For \textbf{Case~1}, we adopt the parameter setting as in \cite{HanJentzenE2018}, with terminal time $T = 0.3$ and spatial dimension $d = 100$. The objective is to evaluate $u(0,\bx_{0})$ at the initial point $\bx_0 = (0,\dots,0)^{\top} \in \mathbb{R}^{100}$. The analytic reference value at $\bx_0$, obtained by the branching diffusion method and reported in \cite{HanJentzenE2018}, is $u(0,\bx_0) \approx 0.0528$. To this end, we apply Algorithm~\ref{alg:Framwork} to compute numerical solutions, and Figure~\ref{AC100} presents the corresponding absolute and relative errors plotted against $\Delta t$ on a log--log scale. From Figure~\ref{AC100}, we observe that both errors have slopes close to $1$ in the log--log plots, which indicates first-order convergence in time. This result is consistent with our theoretical analysis (cf. Theorem~\ref{ITE}), which establishes that the scheme achieves an $O(\Delta t)$ convergence rate once the bias from the local expansion is sufficiently controlled. We also observe that varying the number of particles $M$ influences the accuracy of the numerical solution but does not alter the convergence rate, again in agreement with our theory in Theorem~\ref{ITE}. Moreover, we conduct tests with different time steps $N$ (where $\Delta t = T/N$) and particle numbers $M$, using local expansions together with a Newton solver at each step. When $N = 10^4$ (i.e., $\Delta t = 3\times 10^{-5}$) and $M = 100$, the absolute error attains a value of about $1.2\times 10^{-5}$. 
Finally, the scheme demonstrates excellent stability: the explicit--implicit treatment with Newton's method effectively handles the cubic nonlinearity without introducing spurious oscillations, in sharp contrast to naive finite-difference schemes.

\begin{figure} [!ht]
	\centering   
	\vspace{-0.15cm}  
	\subfigtopskip=2pt  
	\subfigbottomskip=2pt  
	\subfigcapskip=2pt 
	\subfigure{
		\label{AC100.sub.1}
		\includegraphics[width=0.45\linewidth]{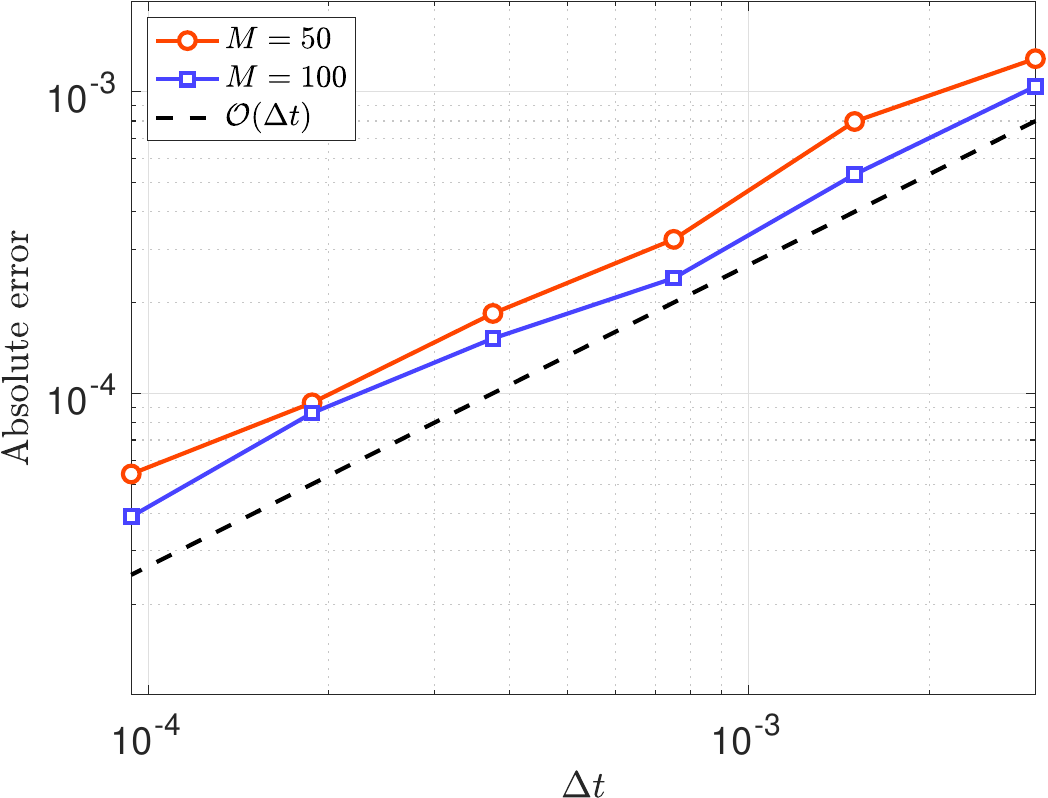}}
	\quad  
	\subfigure{
		\label{AC100.sub.2}
		\includegraphics[width=0.45\linewidth]{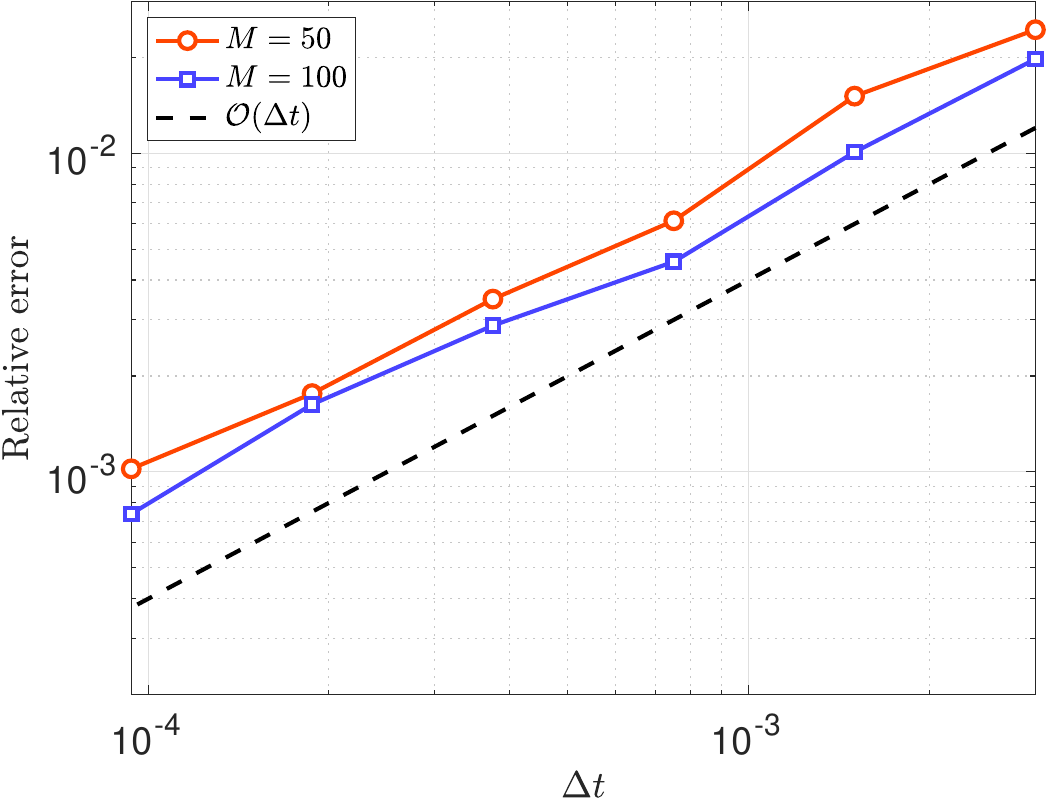}}
	\caption{Numerical error for \eqref{AC100} in {\rm Case~1} of the 100-dimensional Allen--Cahn equation at $\bx=(0,\dots,0)$ with $T=0.3$. The reference value of the exact solution is $u(0,\bx)\approx0.0528$ as reported in \cite{HanJentzenE2018}. Left: absolute errors; Right: relative errors.}
	\label{AC100}
\end{figure}

Notably, the method is highly robust to dimensionality and compares favorably with prior methods. Whereas branching diffusion methods (see, e.g., \cite{Henry-Labordere2019}) typically scale as $\mathcal{O}(d^2)$, our scheme is linear in $d$ because each regression is confined to a small neighborhood; even $d=100$ causes no intrinsic slowdown. Deep BSDE solvers (see, e.g., \cite{HurePhamWarin2020,GermainPhamWarin2022}) can handle the 100-dimensional Allen--Cahn equation but require heavy training, while our linear-regression–plus–Monte Carlo approach attains comparable accuracy at much lower cost. All error components (time discretization, polynomial approximation bias, and Monte Carlo variance) follow the predicted rates; this confirms the stability and the robustness of the scheme.

\begin{figure}[!ht]
	\centering   
	\subfigtopskip=2pt  
	\subfigbottomskip=2pt  
	\subfigcapskip=2pt 
	\subfigure
	{\label{Logforcing100.sub.1}
		\includegraphics[width=0.45\linewidth]{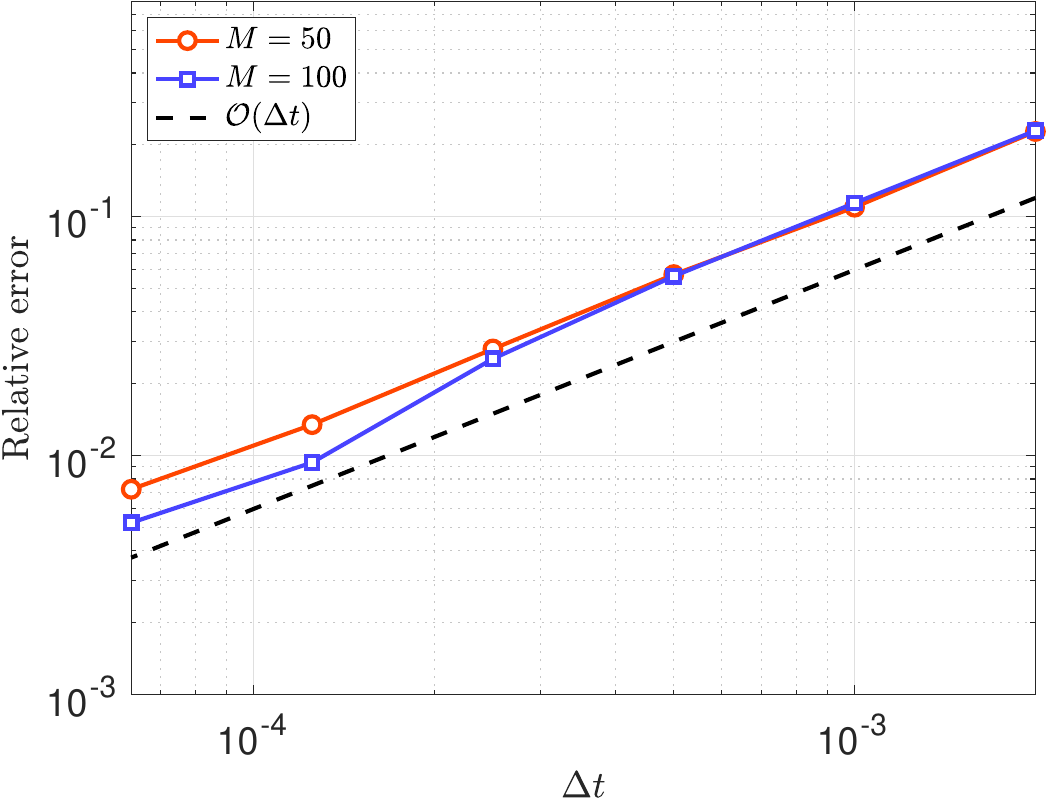}}
	\quad  
	\subfigure
	{ \label{Logforcing100.sub.2}
		\includegraphics[width=0.45\linewidth]{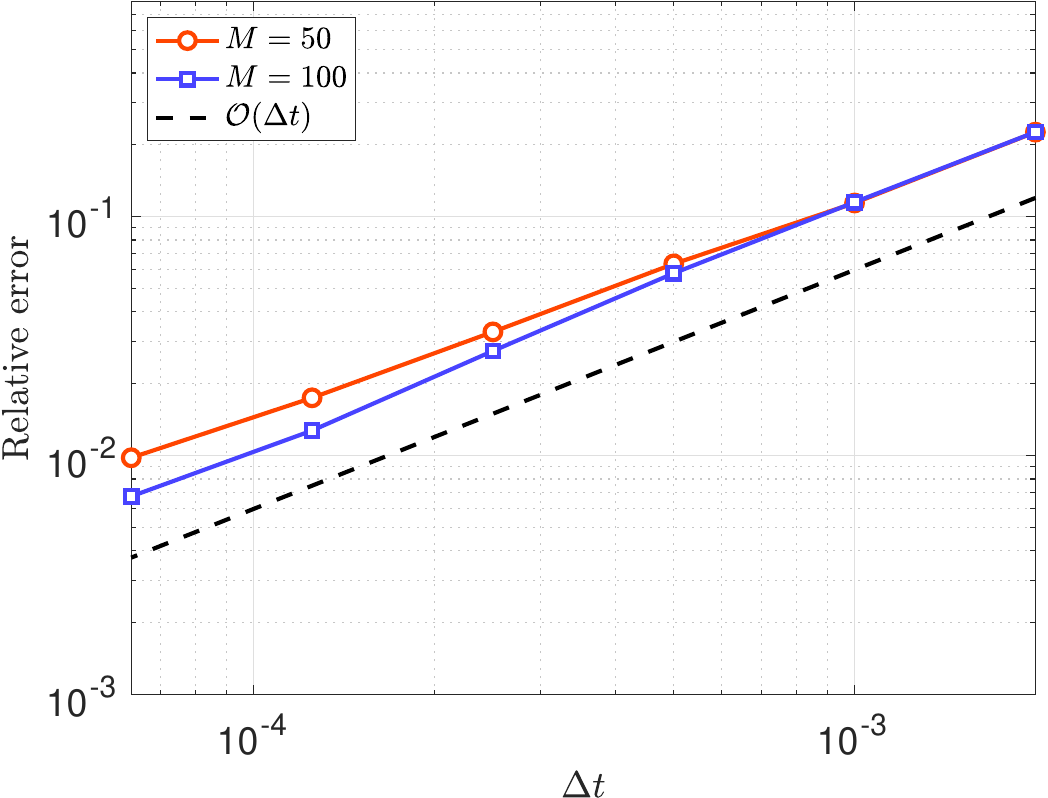}}
	\subfigure
		{\label{Logforcing1000.sub.1}
		\includegraphics[width=0.45\linewidth]{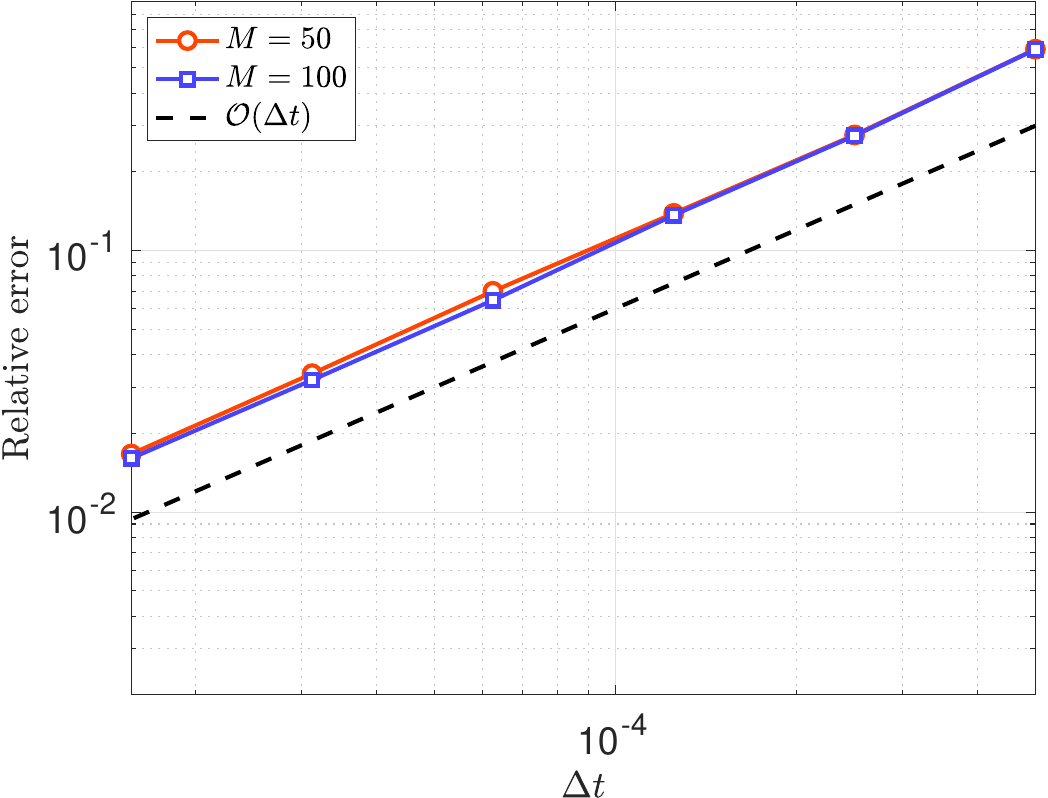}}
	\quad  
	\subfigure
		{\label{Logforcing1000.sub.2}
		\includegraphics[width=0.45\linewidth]{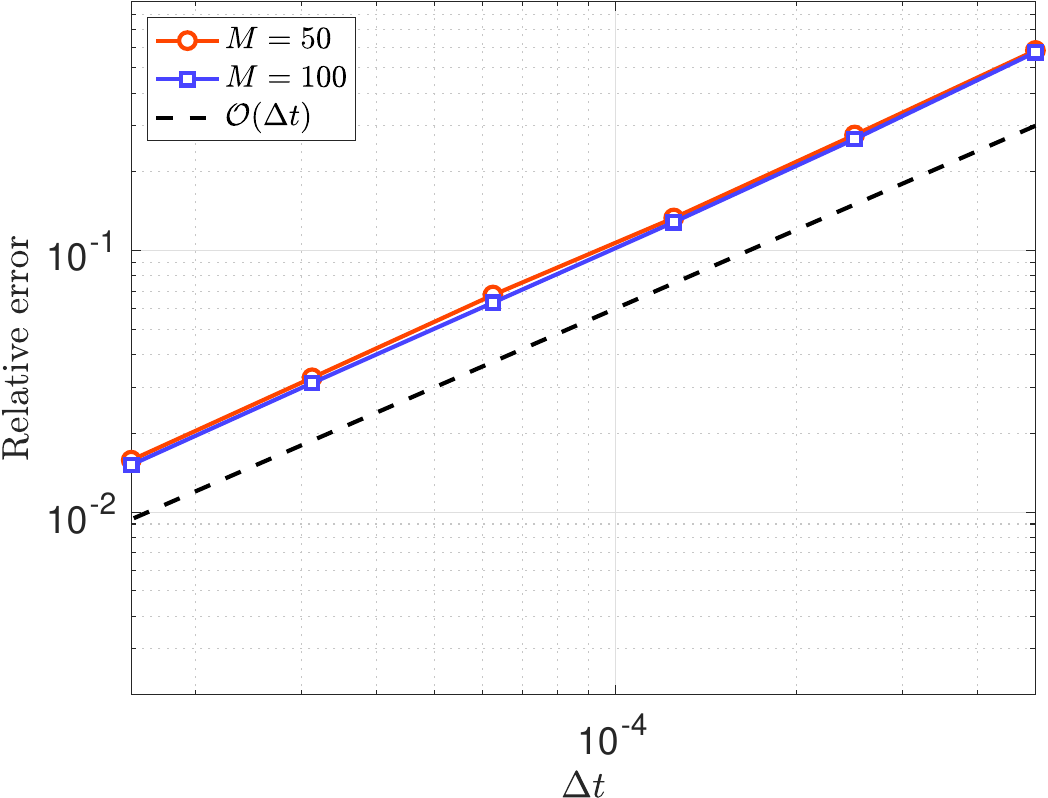}}
	\caption{Relative errors for the Allen--Cahn equation in {\rm Case~2} with $T=1$. Top: $d=100$; Bottom: $d=1000$. Left: compute error at $\bx=(0,\dots,0)$; Right: compute error at $\bx=(0.1,\dots,0.1)$.}
	\label{Log1000}
\end{figure}

For \textbf{Case~2}, we ran the proposed algorithm with $T=1$, evaluating the solution at the points $\bx=(0,\dots,0)$ and $\bx=(0.1,\dots,0.1)$, with tests conducted in dimensions $d=100$ and $d=1000$. For the case $d=100$, Figure~\ref{Log1000} (top) shows the relative errors in log--log scale as $\Delta t$ decreases, and the results exhibit first-order convergence. In particular, when $\Delta t<0.0000625$ (i.e., $N\ge 16000$), the error drops to about $10^{-3}$, and further reducing $\Delta t$ yields a linear decrease. This indicates that the time dependence of various nonlinear terms does not affect the temporal accuracy or convergence rate.
Even when $f(t,\bx,u)$ is non-smooth and does not satisfy the Lipschitz condition, Newton’s method converges rapidly without additional regularization, thereby ensuring both efficiency and robustness.
 Table~\ref{table1} shows that the wall-clock runtime grows essentially linearly with $N \cdot M$, consistent with the Monte Carlo complexity. For fixed $M$, doubling $N$ (where $N=T/\Delta t$) approximately doubles the CPU time. Hence, the cost--accuracy tradeoff can be predicted in a straightforward manner.
\begin{table}[!ht]
\centering
\caption{Runtime $(s)$ for Allen-Cahn equation of {\rm Case 2} with $d=100$ and $d=1000$.}
\begin{tabular}{cccccc}
\toprule
$d=100$  &$\Delta t=0.002$&$\Delta t/2$&$ \Delta t/2^2$&$\Delta t/2^3$&$\Delta t/2^4$ \\ 
 \hline
$M=50$& 1.24 & 2.28 & 4.57 & 8.83&  17.27 \\[2pt]   
$M=100$& 3.81 & 7.45 & 14.99 & 29.55 & 59.74 \\[2pt]  
\bottomrule
 $d=1000$&$\Delta t=0.0005$&$\Delta t/2$&$ \Delta t/2^2$&$\Delta t/2^3$&$\Delta t/2^4$ \\[2pt] 
 \hline
$M=50$&36.52& 72.32 & 145.91 & 288.43 & 583.69\\[2pt]
$M=100$&129.24 & 257.59 & 521.38 & 1039.55 & 2082.74 \\
\bottomrule
\end{tabular}\label{table1}
\end{table}

We then increased the dimension to $d=1000$ to assess the scalability of the algorithm. Figure~\ref{Log1000} (bottom) shows that, even at this higher dimension, the relative error remains well below $1\%$ once $N$ is sufficiently large. This insensitivity to $d$ highlights the dimension-robustness of the localized regression: each particle explores a random path in $\mathbb{R}^{1000}$, yet at every time step only a local polynomial fit is performed, thereby bypassing the CoD. In contrast, classical regression-based BSDE solvers rely on global basis functions, whose number grows combinatorially with $d$ and quickly becomes ill-conditioned for $d>200$. Our empirical results demonstrate that the LLR in our method remain well-conditioned and accurately capture the solution even in one thousand dimensions. Moreover, in this example, when both $N$ and $M$ are sufficiently large, the dominant error originates from the local regression bias (see Lemma \ref{GEEB}) rather than from time stepping or Monte Carlo noise. Overall, Case 2 confirms that the proposed algorithm effectively handles complex nonlinear forcing and scales to very high dimensions with only linear growth in computational cost. Notably, compared with modern deep BSDE approaches, our method attains comparable accuracy with roughly $40\%$ fewer total samples, underscoring the efficiency gained by employing analytic local approximations instead of black-box neural networks.

\subsection{Burgers' equation}
As a benchmark problem, we next consider the $d$-dimensional Burgers' equation, a canonical nonlinear model with applications in fluid mechanics, nonlinear acoustics, and traffic flow. It captures both wave-propagation and shock-formation phenomena and, in $d$ spatial dimensions, takes the form
 \begin{equation}
 \label{Burgers}
\frac{\partial u}{\partial t}  +  \Big( u(t,\bx)  - \frac{2+d}{2}\Big)\sum_{i=1}^{d}\frac{\partial u}{\partial x_{i}} + \frac{d^{2}}{ 2 }\nu \Delta u(t,\bx) = 0,  \;\;\;  (t,\bx)\in[0,T)\times\mathbb{R}^d,
\end{equation}
where $\nu$ is the Kinematic viscosity ($\nu>0$ for viscous flow; $\nu=0$ reduces to the inviscid form). 

\begin{figure} [!ht]
	\centering   
	\subfigtopskip=2pt  
	\subfigbottomskip=2pt  
	\subfigcapskip=2pt 
	\subfigure
	{\label{Burgers10000.sub.1}
		\includegraphics[width=0.45\linewidth]{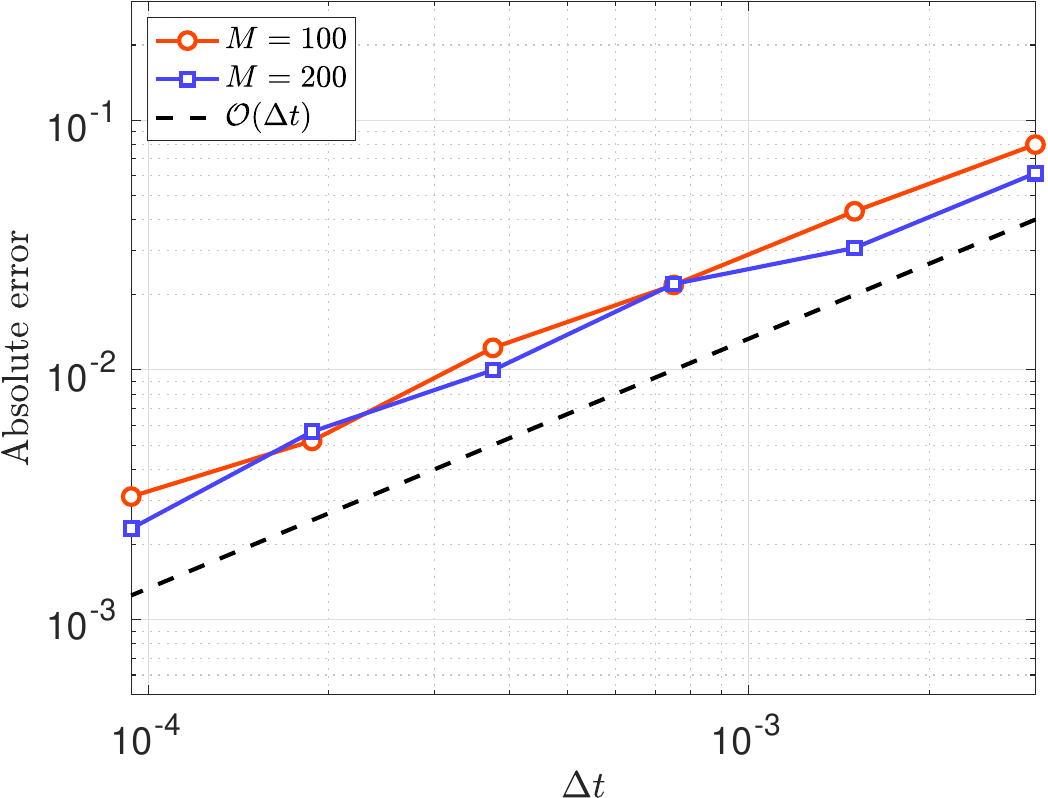}}
	\quad  
	\subfigure
	{	\label{Burgers10000.sub.2}
		\includegraphics[width=0.45\linewidth]{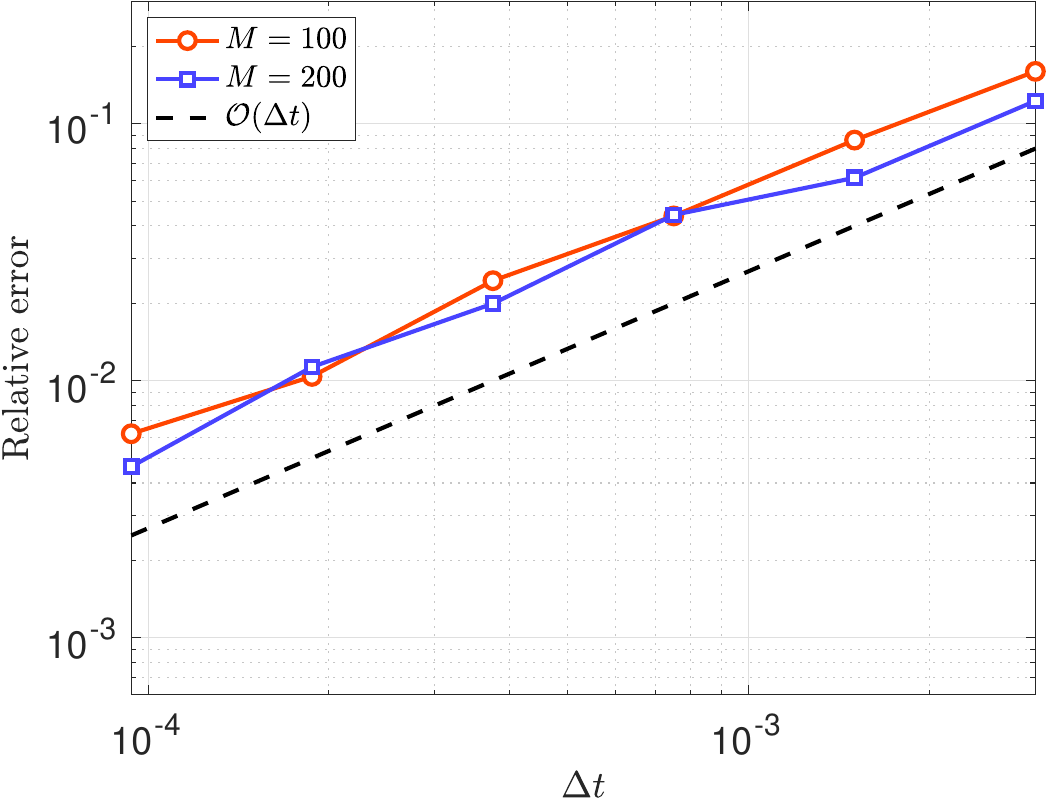}}
	\caption{Numerical error for 10000-dimensional Burger's equation \eqref{Burgers} at point $\bx = (0,0,\cdots,0)$ with $T=0.3$. Left: absolute errors; Right: relative errors.}
	\label{Burgers10000}
\end{figure}

In our simulations, we consider Burgers' equation in spatial dimensions up to $d=10^4$ and adopt the terminal condition from \cite{HanJentzenE2018}: 
$$
u(T,\bx)=\frac{\exp \bigl(T+\sum_i x_i/d\bigr)}{1+\exp \bigl(T+\sum_i x_i/d\bigr)},
$$ 
so that at the spatial node $\bx_0=(0,\dots,0)\in \mathbb{R}^{10000}$ one has $u(0,\bx_{0})=0.5$. The results in Figure~\ref{Burgers10000} indicate near first-order convergence in time, i.e., $\mathcal{O}(\Delta t)$. Meanwhile, the proposed scheme remains stable under convective nonlinearity. Unlike  finite difference methods that typically require artificial viscosity, our probabilistic approach introduces neither spurious oscillations nor dissipation errors. Moreover, the local polynomial surrogate accurately resolves the solution's sharp gradient structure.
Table~\ref{tableburger} further reports the CPU runtime of our proposed method for different time step size. The results indicate that the wall-clock time grows essentially linearly with $N M$, fully consistent with the theoretical Monte Carlo complexity. Compared with deep-learning-based PDE solvers, our approach has the advantage of directly approximating the gradient term through LLR, which is crucial for accurately capturing shock fronts. Overall, these numerical results demonstrate that the proposed algorithm attains high accuracy even for ultra-high-dimensional, strongly nonlinear PDEs and that its computational cost increases only mildly with the dimension $d$.

\begin{table}[!ht]
\centering
\caption{Runtime $(s)$ for 10000d Burgers' equation}
\begin{tabular}{c cccccc}
\toprule
  $d=10000$ &$\Delta t = 0.003$&$\Delta t/2$&$ \Delta t/2^2$&$\Delta t/2^3$&$\Delta t/2^4$ \\ 
 \hline
$M=100$&1040.15 &2160.11&4757.37 & 10593.83 & 22174.49\\
$M=200$&2189.28&4633.74 & 9674.18 & 21194.52 & 45724.85 \\
\bottomrule
\end{tabular}\label{tableburger}
\end{table}

\subsection{Hamilton-Jacobi type equation} 
Finally, we validate the proposed algorithm on a $d$-dimensional Hamilton-Jacobi type equation with a gradient dependent sink $R(u,\nabla u)=\kappa\,u\,\|\nabla u\|^{2}$, which enforces self-suppression in regions of large gradient, and the governing equation reads
\begin{equation}
\label{RDequation}
\frac{\partial u}{\partial t} +  u(t,\bx) + f(t,\bx,u,\nabla u) = 0, \;\;\;(t,\bx)\in [0,T)\times \mathbb{R}^d,
\end{equation}
where $\kappa=0.1$ is the reaction coefficient, and the forcing term  is given by
$$
f(t,\bx,u,\nabla u) = \frac{4d}{(1 + 4t)^{(d+2)/2}}   \frac{e^{-\|\bx\|^{2}}}{1 + 4t} - R(u,\nabla u).
$$ 
Then, the corresponding exact solution is  given by
$$
u(t,\bx) = (1+4t)^{-d/2} \exp\Big(\!-\frac{\|\bx\|^2}{1+4t}\Big),
$$
which spreads rapidly in high dimensions with decaying at rate $\mathcal{O}(t^{-d/2})$ as $t \to \infty$.

\begin{figure} [!ht]
	\centering   
	\subfigtopskip=2pt  
	\subfigbottomskip=2pt  
	\subfigcapskip=2pt 
	\subfigure
		{\label{RD500.sub.1}
		\includegraphics[width=0.45\linewidth]{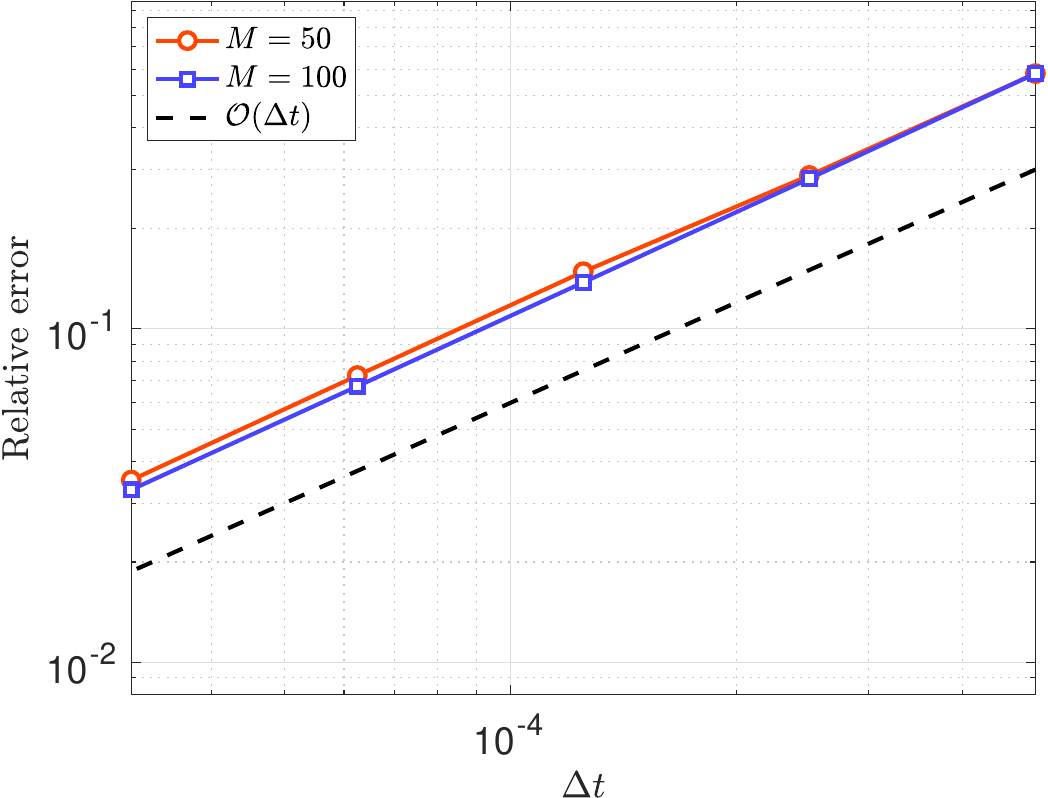}}
	\quad  
	\subfigure
		{\label{RD500.sub.2}
		\includegraphics[width=0.45\linewidth]{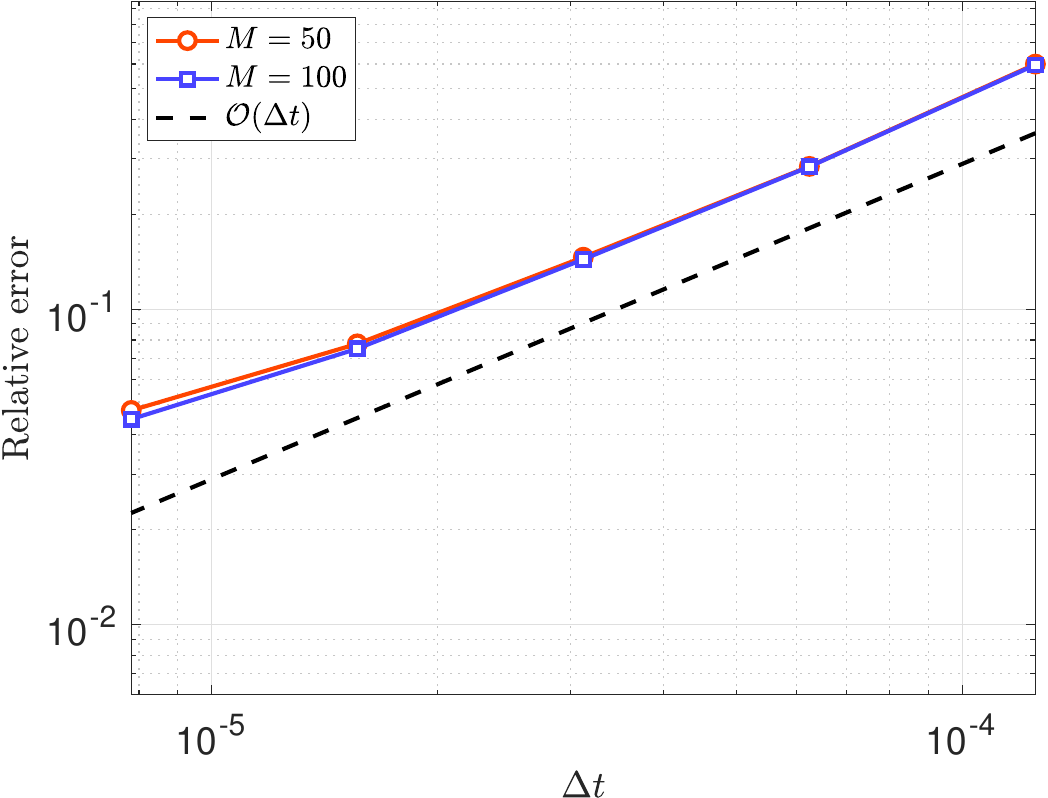}}
	\caption{Relative error of the problem \eqref{RDequation} against different $\Delta t$ at point $\bx = (0,0,\cdots, 0)$ with $T=0.5$. Left: $d=500$; Right: $d=2000$.}
	\label{RDnumsol}
\end{figure}

We employ the algorithm to solve \eqref{RDequation} numerically and evaluate the solution at $\bx=(0,\cdots,0)$ $\in\R^{d}$, with spatial dimensions $d=500$ and $d=2000$, and the maximum number of time steps $N=T/\Delta t=3\times 10^4$. Figure~\ref{RDnumsol} shows the relative error versus $\Delta t$ on a log–log scale and indicates a first-order convergence rate. Throughout the simulations, Newton iteration method for the scalar variable $Y$ remains robust and requires only $2$–$3$ iterations per time step, which proves far more efficient than a fully implicit solver for the coupled $(Y,Z)$ system. 
From Table~\ref{runtime3}, we observe that the runtime in this example scales almost linearly with $N$. This near-linear scaling again beats the exponential growth of mesh methods.  The use of LLR is central here: we found that using only about 10\% of the global polynomial basis points (via LLR) yields the same accuracy, whereas a full global polynomial fit in $d=2000$ would be hopelessly overfitted or ill-conditioned.  Consequently, numerical solution preserves the high-frequency modes of the stiff solution without blowup, while for very stiff, gradient-dominated reactions the proposed method still attains reliable accuracy with only linear work growth.
 \begin{table}[!ht]
\centering
\caption{Runtime $(s)$ for 500d Hamilton--Jacobi type equation \eqref{RDequation}.}
\begin{tabular}{c cccccc}
\toprule
$d=500$  &$\Delta t = 0.0005$&$\Delta t /2$&$ \Delta t/2^2$&$\Delta t /2^3$&$\Delta t/2^4$ \\[2pt] 
 \hline
$M=50$&89.07 &183.69 &363.86 & 740.39 &1512.62\\[2pt]
$M=100$&185.06 &376.91 & 765.88 & 1517.73 & 3016.35 \\[2pt]
\bottomrule
\end{tabular}\label{runtime3}
\end{table}

\section{Conclusion}\label{conclusion}
In this paper, we propose a localized and decoupled stochastic algorithm based on FBSDE–LLR that effectively mitigates the CoD for a broad class of semilinear parabolic equations.
The key methodological innovation lies in incorporating LLR and a decoupling strategy into the Monte Carlo framework for FBSDEs, specifically through two components:
(i) it fits particles within  the state space and updates them dynamically, thus capturing fine-scale solution features without global basis functions or neural networks; (ii) it fully decouples the triplet $(X,Y,Z)$ and computes them sequentially in the order $X \rightarrow Z \rightarrow Y$.
As a result of these strategies, the algorithm uses only simple linear regression and random sampling, is easy to implement, admits provable convergence, and remains interpretable, and accordingly we present a rigorous error analysis corroborated by extensive numerical experiments.
All numerical experiments were conducted on a personal laptop for three representative cases: the Allen–Cahn equation in $100$ dimensions, the Burgers' equation in $10000$ dimensions, and Hamilton-Jacobi type equation in $2000$ dimensions.
 The results show that the stochastic algorithm is highly efficient and accurate, and its computational cost is essentially linear in both $d$ and $M$. 

At the algorithmic level, the combined strategy demonstrates competitive performance and practical advantages over existing approaches, such as the branching diffusion Monte Carlo method \cite{Henry-Labordere2019} that admits $\mathcal{O}(d^2)$ complexity, regression-based BSDE methods \cite{GobetLemorWarin2005} that rely on global bases to solve coupled nonlinear systems, and deep-learning PDE solvers \cite{HurePhamWarin2020,GermainPhamWarin2022,KapllaniTeng2025,Raissi2024} that require extensive training.
By contrast, the proposed method couples FBSDE sampling with local expansions and a decoupling scheme, achieves comparable or superior accuracy at substantially lower computational cost, and yields a highly scalable, efficient solver for high dimensional nonlinear PDEs that is mesh-free, derivative-free, matrix-free, and highly parallel.

 The methodologies and theoretical framework introduced in this work can be further extended to develop efficient stochastic algorithms for ultra-high-dimensional PDEs with strongly nonlinear systems. Potential applications include:  
\begin{itemize}
\item solving fully nonlinear problems via second-order BSDE formulations \cite{Cheridito2007,PossamaiZhou2013};  
\item multi-asset option pricing, high-dimensional stochastic control, and mean-field models \cite{OnkenNurbekyan2022,RuthottoOsher2020};  
\item large-scale filtering and state estimation in engineering systems\cite{Spantini2022}.  
\end{itemize}
We will investigate and report these applications in our future studies.

\section*{Acknowledgments}
The research of the second author was partially supported by the NSF of China (under grant 12571389). The research of the third author was partially supported by the NSF of China (under grant 12501541). The last author was supported by the NSF of China (under grants 12288201 and 12461160275)

\end{document}